\documentclass[10pt,a4paper]{article}
\usepackage{amsmath}
\usepackage{amsthm}
\usepackage{amsfonts}
\usepackage{amssymb}
\usepackage{indentfirst}

\newtheorem{thm}{Theorem}[section]
\newtheorem{lem}[thm]{Lemma}

\newtheorem{pr}[thm]{Proposition}
\newtheorem{re}{Remark}[section]
\numberwithin{equation}{section}

\newcommand{\N}{\mathbb{N}}

\newcommand{\E}{\mathbb{E}}

\newcommand{\PP}{\mathbb{P}}

\begin{document}
\title{Convergence of martingale and moderate deviations for a branching random walk with a  random environment in time }
\author{Xiaoqiang Wang$^{a}$, Chunmao Huang$^{b}$
\\
\small{\emph{$^{a}$Shandong university (Weihai), School of Mathmatics and Statistics, 264209, Weihai, China}}\\
\small{\emph{$^{b}$Harbin institute of technology at Weihai, Department of mathematics, 264209, Weihai, China}}}
\maketitle

\begin{abstract}
 We consider a branching random walk on $\mathbb{R}$ with a stationary and ergodic environment $\xi=(\xi_n)$ indexed by time $n\in\mathbb{N}$.  Let $Z_n$ be the counting measure of particles of generation $n$ and $\tilde Z_n(t)=\int e^{tx}Z_n(dx)$ be its Laplace transform. We show the $L^p$ convergence rate and the uniform convergence of the martingale $\tilde Z_n(t)/\mathbb E[\tilde Z_n(t)|\xi]$, and establish a moderate deviation principle for the measures $Z_n$.


\end{abstract}

\section{Introduction}\label{LTS1}

\subsection{Model and notation}

\noindent \emph{Branching random walks} were largely studied in the literature, see e.g. \cite{biggins77, biggins, biggins79, biggins97, b, chauvin,ka}.  In the classical branching random walk, the point processes indexed by the particles $u$, formulated by the the number of its offsprings and their displacements, have a common distribution for all particles. However, in reality these distributions may differ from generations according to an environment in time, or depend on particles' positions according to an environment in space. For this reason, branching random walks in random environments attract many authors' attention recently. Many results for classical branching random walk have been extended to random environments  both in time and space, see e.g. \cite{ CP, CY, greven, huang2,MN, yosh}.
Here we consider the case in a time random environment, where the distributions of the point processes indexed by  particles vary from generation to generation according to a random environment in time. Such a model is called \emph{branching random walk with a random environment in time} (BRWRE). It was first introduced by Biggins \& Kyprianou \cite{biggins04}. 
Recently, some limit theorems such as large deviation principles and central limit theorems were obtained in \cite{gao14, huang2,huang22}. %

\bigskip

 Let's describe the model. The random environment in time is modeled as a stationary and ergodic
sequence  of random variables, $\xi=(\xi_n)$, indexed by the time $n\in\mathbb{N}=\{0,1,2,\cdots\}$, taking values in some measurable space $(\Theta, \cal E)$.  Without loss of generality we can suppose that $\xi$ is defined on the product space
$(\Theta^{\N}, \cal E^{\otimes\N}, \tau) $, with  $\tau$  the law of $\xi$. Each realization of $\xi_n$ corresponds to a distribution $\eta_n=\eta(\xi_n)$
  on $\mathbb{N}\times\mathbb{R}\times\mathbb{R}\times\cdots$.
When the environment $\xi=(\xi_n)$ is given, the process can be described as follows. At time $0$, there is an initial particle $\emptyset$ of generation $0$ located at $S_{\emptyset}=0\in\mathbb{R}$;
at time $1$, it is replaced by $N=N({\emptyset})$ particles of generation $1$, located at $L_i=L_i({\emptyset })$, $1\leq i\leq N$, where the random vector
$X({\emptyset})=(N, L_1, L_{2},\cdots)\in\mathbb{N}\times\mathbb{R}\times\mathbb{R}\times\cdots$ is of distribution $\eta_0=\eta(\xi_0)$.  
In general, each particle $u=u_1\cdots u_n$ of generation $n$ located at $S_u$ is replaced at
time {n+1} by $N(u)$ new particles $ui$ of generation $n+1$, located at
$$S_{ui}=S_u+L_i(u)\qquad(1\leq i\leq N(u)),$$
where the  random vector $X(u)=(N(u),L_1(u),L_2(u),\cdots)$ is of distribution $\eta_n=\eta(\xi_n)$.  Note that the values $L_i(u)$ for $i> N(u)$ do not play any role for our model; we introduce them only for convenience. We can for example take $L_i(u)=0$ for $i> N(u)$. All particles behave independently conditioned on the environment $\xi$.

For each realization $\xi \in \Theta^\N$ of the environment sequence,
let $(\Gamma, {\cal G},  \mathbb{P}_\xi)$ be the probability space under which the
process is defined. 
The probability
$\mathbb{P}_\xi$ is usually called\emph{ quenched law}.
The total probability space can be formulated as the product space
$( \Theta^{\mathbb{N}}\times\Gamma , {\cal E}^{\N} \otimes \cal G,   \mathbb{P})$,
 where $ \mathbb{P} = \E  (\delta_\xi \otimes \mathbb{P}_{\xi}) $ with $\delta_\xi $ the Dirac measure at $\xi$ and $\E$ the expectation with respect to the law of  $\xi$, so that  for all measurable and
 positive function $g$ defined on $\Theta^{\mathbb{N}}\times\Gamma$, we have
  $$\int_{ \Theta^{\mathbb{N}}\times\Gamma } g (x,y) d\mathbb{P}(x,y) = \E  \int_\Gamma g(\xi,y) d\mathbb{P}_{\xi}(y).$$
The total
probability $\mathbb{P}$ is usually called \emph{annealed law}.
The quenched law $\mathbb{P}_\xi$ may be considered to be the conditional
probability of $\mathbb{P}$ given $\xi$. The expectation with respect to $\mathbb{P}$ will still be denoted by $\E$; there will be no confusion for reason of consistence.   The expectation with respect to
$\PP_\xi$ will be denoted by $\E_\xi$.


\medskip

 Let $$\mathbb{U}=\{\emptyset\}\bigcup\limits_{n\geq1}\mathbb{N}^n$$ be the set of all finite sequence $u=u_1\cdots u_n$. By definition, under $\mathbb{P}_\xi$, the random vectors $\{X(u)\}$, indexed  by $u\in\mathbb{U}$, are independent of each other, and each  $X(u)$ has distribution $\eta_n=\eta(\xi_n)$ if $|u|=n$, where $|u|$ denotes the length of $u$.
 Let $\mathbb{T}$ be the Galton-Watson tree  with defining element $\{N(u)\}$. We have: (a) $\emptyset\in\mathbb{T}$; (b) if $u\in\mathbb{T}$, then $ui\in\mathbb{T}$ if and only if $1\leq i\leq N(u)$; (c) $ui\in\mathbb{T}$ implies $u\in\mathbb{T}$.
 Let  $\mathbb{T}_n=\{u\in\mathbb{T}: |u|=n\}$ be the set of particles of generation $n\in\mathbb N$ and
\begin{equation}Z_n(\cdot)=\sum_{u\in\mathbb{T}_n}\delta_{S_u}(\cdot) \end{equation}
be the counting measure of particles of generation $n$. For a measurable  subset $A$ of $\mathbb R$,
$Z_n(A)$
denotes the number of particles of generation $n$ located in $A$. For any finite sequence $u$, let \begin{equation}X^{(u)}(\cdot) = \sum_{i=1}^{N(u)}   \delta_{L_i (u)} (\cdot)  \end{equation} be the counting measure  corresponding to the random vector $X(u)$, whose increasing points are $L_i(u)$,  $1\leq i \leq N(u)$. Denote $u|n$ by the restriction to the first $n$ terms of $u$, with the convention that   $u_0|0=\emptyset$. Set
\begin{equation}X_n=X^{(u_0|n)},\end{equation} where $u_0=(1,1,\cdots)$. The counting measure $X_n$ describes the evolution of the system at time $n$.

 For $n\in \mathbb{N}$ and $t\in \mathbb{R}$, denote
\begin{equation}
\tilde Z_n (t) = \int e^{tx}Z_n (dx)  = \sum_{u\in\mathbb{T}_n}
 e^{tS_u}
 \end{equation}
the Laplace transform of $Z_n$. It is also called \emph{partition
 function} by physicians.  In particular, for $t=0$, $\tilde Z_n(0)=Z_n(\mathbb R)$. Let
\begin{equation}
m_n(t) = \mathbb{E}_\xi\int e^{tx}X_n(dx)=  \mathbb{E}_\xi \sum_{i=1}^{N(u)} e^{tL_i(u)} \quad(|u|=n),
\end{equation}
be the Laplace transform of the counting measure describing the
evolution of the system at time $n$. Put
 \begin{equation}
P_0(t)=1\qquad \text{and}\qquad P_n(t)=\prod_{i=0}^{n-1}m_i(t) \quad\text{for $n\geq1$.}
\end{equation}
Then $P_n(t)=\mathbb{E}_\xi \tilde Z_n(t)$. Moreover, set
\begin{equation}
\tilde X_u(t)=\frac{e^{tS_u}}{P_{n}(t)}\qquad(|u|=n),
\end{equation}
and
\begin{equation}
W_n(t)=\frac{\tilde Z_n(t)}{P_n(t)}=\sum_{u\in\mathbb T_n}\tilde X_u(t).
\end{equation}
 Let $\mathcal {F}_0=\sigma(\xi_0,\xi_1,\xi_2,\cdots)$ and $\mathcal
{F}_n=\sigma(\xi_0,\xi_1,\xi_2,\cdots,X(u),\; |u|<n,\;i=1,2,\cdots)$. It is well known that for each $t$ fixed, $W_n(t)$ forms a
nonnegative martingale with respect to the filtration ${\mathcal {F}_n}$ under both laws $\mathbb{P}_\xi$ and  $\mathbb{P}$,  and
\begin{equation}
\lim_{n\rightarrow\infty}W_n(t)=W(t)\qquad a.s.
\end{equation}
with $\mathbb{E}_\xi W(t)\leq1$.  In the deterministic environment case, this martingale has been
studied by Kahane \&  Peyri\`ere \cite{kahane}, Biggins \cite{biggins77},
Durrett \& Liggett \cite{durrett}, Guivarc'h \cite{guivarch}, Lyons \cite{lyons} and
Liu \cite{liu97, liu98,liu2000,liu01}, etc. in different contexts.

Assume throughout that
\begin{equation}\label{ASS}
\mathbb{E}\log m_0(0)\in(0,\infty) \qquad\text{and} \qquad \mathbb E\left[\frac{N}{m_0(0)}\log^+N\right]<\infty.
\end{equation}
The first condition means that the corresponding  branching process in a random environment (BPRE), $\{Z_n(\mathbb R)\}$, is \emph{supercritical}, so that the survival of the population $\{Z_n(\mathbb R)\rightarrow\infty\}$ has positive probability; and the two conditions ensure that the limit of the normalized population, $W(0)$, is non-degenerate (cf. \cite{a1,a}). We also assume that
\begin{equation} \label{H1}
 \mathbb{E}|\log m_0 (t)|<\infty\qquad \mbox{ and } \qquad\mathbb{E}\left|\frac{m_0' (t)}{m_0 (t)} \right|<\infty
\end{equation}
for all $t\in \mathbb{R}$. The last two moment conditions imply  that
\begin{equation} \Lambda (t) := \mathbb{E}\log m_0 (t) \qquad \mbox{ and } \qquad \Lambda' (t) =  \mathbb{E} \frac{m_0' (t)}{m_0 (t)}
\end{equation}
are well defined as real numbers, so that $\Lambda (t)$
is differentiable everywhere on $\mathbb{R}$ with $\Lambda' (t)$ as its derivative.
Let
\begin{equation}\label{HT}
\begin{array}{l}
t_-= \inf \{ t \in \mathbb{R}: t\Lambda'(t) - \Lambda (t) \leq 0\}, \\
t_+= \sup \{ t \in \mathbb{R}: t\Lambda'(t) - \Lambda (t) \leq 0\},
\end{array}
\end{equation}
Then $ -\infty \leq t_- <0 < t_+ \leq \infty$, $t_-$ and $t_+$ are
two solutions of $t\Lambda'(t) - \Lambda (t) = 0$ if they are
finite. Denote
$$I=(t_-, t_+).$$
If $t\in I$ and $\mathbb E W_1(t)\log^+W_1(t)<\infty$, then $\mathbb E_\xi W(t)=1$ a.s. (cf. \cite{biggins04, ku}).

\bigskip
\subsection{$L^p$ convergence rate}

We first study the $L^p$ $(p>1)$ convergence of $W_n(t)$ to its limit $W(t)$ and its exponential  rate for $t\in\mathbb{R}$ fixed. When $t=0$, $W_n(0)$ reduces to the normalized population of the corresponding BPRE, whose convergence rate is carefully discussed in Huang \& Liu \cite{huang14}.  Without loss of generality, here we only consider the case where $t=1$ and assume that
\begin{equation}\label{LPH1}
m_0(1)=1.
\end{equation}
Write 
$W_n=W_n(1)$ for short.
For general case,  if $m_0(t)\in(0,\infty)$ a.s., we can construct a new BRERE with relative displacements  $\bar L_i(u)=tL_i(u)-\log m_n(t)$ $(|u|=n)$. Then this new BRERE satisfies $\bar m_0(1)=1$ and $\bar W_n=W_n(t)$.
Furthermore,  we also assume that
\begin{equation}\label{LPH2}
\mathbb{P}(W_1=1)<1,
\end{equation}
which avoids the trivial case where $W_n=1$ a.s..

\medskip

The following theorem shows the $L^p$ convergence (with $\rho=1$) of $W_n$ under quenched law $\mathbb P_\xi$ and its exponential rate (with $\rho>1$).

\begin{thm}[Quenched $L^p$ convergence rate]\label{LPBRW1.2}
Assume (\ref{LPH1}). Let $p>1$ and $\rho\geq1$.
\begin{itemize}
\item[(a)] If $1<p<2$,
$$\mathbb{E}\log\mathbb{E}_\xi W_1^r<\infty\qquad \text{and}\qquad \rho<\exp(-\frac{1}{r}\mathbb{E}\log m_0(r))$$
for some $r\in[p,2]$, then
$$W_n-W=o(\rho^{-n})\quad \text{a.s.  and in $\mathbb P_\xi$-$L^p$ for almost all $\xi$.}$$
\item[(b)] If $p=2$ or $2<p\leq t_+$, and $\mathbb E \log \mathbb E_\xi W_1^p<\infty$, then for for almost all $\xi$,
$$\limsup_{n\rightarrow\infty}\rho^n(\mathbb{E}_\xi|W_n-W|^p)^{1/p}\left\{\begin{array}{ll}
=0\quad  &\text{if}\;\rho<\rho_c,\\
>0\quad   &\text{if}\;\rho>\rho_c \;\;\text{and}\;\; \mathbb E\log^-\mathbb E_\xi |W_1-1|^2<\infty,
\end{array}
\right.$$
where $\rho_c=\exp(-\frac{1}{2}\mathbb{E}\log m_0(2))$.
\end{itemize}
\end{thm}

If $p=2$ or $2<p\leq t_+$, Theorem \ref{LPBRW1.2}(b) shows  that under certain moment conditions,  the value $\rho_c$ is the critical value for the $L^p$ convergence of $\rho^n(W_n-W)$ to $0$ under quenched law $\mathbb{P}_\xi$.
In order to show the $L^p$ convergence of $W_n$ under annealed law $\mathbb{P}$ and its  rate, we need to assume that the environment random variables $(\xi_n)$ are i.i.d..

\begin{thm}[Annealed $L^p$ convergence]\label{LPBRW1.3}
Assume (\ref{LPH1}), (\ref{LPH2}) and that $(\xi_n)$ are i.i.d..
Let $p>1$.  Then $W_n\rightarrow W$ in $\mathbb P$-$L^p$ if and only if
$$ \mathbb{E}W_1^p<\infty\qquad \text{and}\qquad\mathbb{E}m_0(p)<1.$$
\end{thm}

Theorem \ref{LPBRW1.3} coincide with a result of Liu \cite{liu2000} on branching random walk in a deterministic environment, and is an extension of a result of Guivarc'h \& Liu \cite{liu1} on branching process in a random environment. The same result is obtained in \cite{huang2} with a different approach.

For the exponential rate of the annealed $L^p$ convergence of $W_n$, we have the follow result.

\begin{thm}[Annealed $L^p$ convergence rate]\label{LPBRW1.4}
Assume (\ref{LPH1}), (\ref{LPH2}) and that $(\xi_n)$ are i.i.d..
Let $p>1$ and $\rho>1$.
\begin{itemize}
\item [(a)] If  $1<p<2$, $$ \mathbb{E}(\mathbb{E}_\xi W_1^r)^{p/r}<\infty\qquad \text{and}\qquad \rho<[\mathbb{E}m_0(r)^{p/r}]^{-1/p}$$
for some $r\in[p,2]$, then
$$W_n-W=o(\rho^{-n})\qquad \text{in $\mathbb P$-$L^p$.}$$
\item [(b)] If $p\geq2$ and  $\mathbb{E}W_1^p<\infty$, then
$$\limsup_{n\rightarrow\infty}\rho^n(\mathbb{E}|W_n-W|^p)^{1/p}\left\{\begin{array}{ll}
=0\quad  &\text{if}\;\rho<\rho_0,\\
>0\quad   &\text{if}\;\rho>\rho_0,
\end{array}
\right.$$
where $\rho_0=\min\{[\mathbb{E}m_0(p)]^{-1/p},[\mathbb{E}m_0(2)^{p/2}]^{-1/p}\}$.
\end{itemize}
\end{thm}

Under the conditions of  Theorem \ref{LPBRW1.4}, we can also obtain $W_n-W=o(\rho^{-n})$ a.s. in $\mathbb P_\xi$-$L^p$ for almost all $\xi$. However,  by  Jensen's inequality, one can see that $\frac{r}{p}\mathbb{E}\log\mathbb{E}_\xi W_1^r\leq\log \mathbb{E}(\mathbb{E}_\xi W_1^r)^{p/r}$ and $[\mathbb{E}m_0(r)^{p/r}]^{-1/p}\leq\exp(-\frac{1}{r}\mathbb{E}\log m_0(r))$.
So the moment conditions of Theorem \ref{LPBRW1.2} are weaker than those of Theorem \ref{LPBRW1.4}. If $p\geq2$, Theorem \ref{LPBRW1.4} (b) shows that under the moment condition $\mathbb{E}W_1^p<\infty$, the value $\rho_0$ defined above is the critical value for the $L^p$ convergence of $\rho^n(W_n-W)$ to $0$ under annealed law $\mathbb{P}$. Obviously, we have $\rho_0\leq \rho_c$. But here it is a pity that  we do not find the critical value for $p\in(1,2)$, in contrast to \cite{huang14} for branching process in a random environment.

\bigskip
\subsection{Uniform convergence}

We next consider the uniform convergence of  $W_n(t)$ to its limit $W(t)$. In  the deterministic environment case, such result was shown by Biggins \cite{b91, b92} and recently is generalized  by Attia \cite{na}.

Recall that $I=(t_-,t_+)$, where $t_+$ and $t_-$ are defined by (\ref{HT}).
If $t\in I$ and $\mathbb E W_1(t)\log^+W_1(t)<\infty$, then $ W(t)$ is non-degenerate. Similarly to \cite{b91, b92, na}, we consider the  uniform convergence of  $W_n(t)$ on subsets of $I$. Denote
$$
\underline m_0=\inf\limits_{t\in I}m_0(t),
$$
\begin{equation}\label{O1}
\Omega_1=int\{t\in\mathbb R: \mathbb E\log \mathbb E_\xi W_1(t)^\gamma \;\; \text{for some $\gamma >1$}\},
\end{equation}
$$\Omega_2=int\{t\in\mathbb R: \mathbb E \tilde Z_1(t)\log^+\tilde Z_1(t)<\infty\}.$$
Here and after we use the following usual notations:
\[\log^+x=\max(\log x, 0), \quad\log^-x=\max(-\log x, 0). \]

\begin{thm}[Quenched uniform convergence]\label{LPBRWU1}Assume that $\mathbb E\log^-\underline m_0<\infty$.
\begin{itemize}
\item[(a)] Let $K$ be a compact subset of $I\bigcap\Omega_1$. Then there exists a constant $p_K\in(1,2]$ such that $W_n(t)$ converges uniformly to $W(t)$ on $K$, almost surely and in $\mathbb P_\xi$-$L^p$ for almost all $\xi$ if $p\in[1,p_K]$.
\item[(b)]If $\underline m:=essinf \underline m_0>0$, then $W_n(t)$ also converges uniformly to $W(t)$ on any compact subset $K$ of $I\bigcap\Omega_2$, almost surely and in $\mathbb P_\xi$-$L^1$ for almost all $\xi$.
\end{itemize}
\end{thm}

If the environment random variables $(\xi_n)$ are i.i.d., Theorem \ref{LPBRWU1} have the following comparison under annealed law.

\begin{thm}[Annealed uniform convergence]\label{LPBRWU2}Assume that $\underline m=essinf \underline m_0>0$. Denote
$$I'=int\{t\in\mathbb R: \mathbb E\left[\frac{m_0(\gamma t)}{m_0(t)^\gamma}\right]<1\;\;\text{for some $\gamma >1$}\}.$$
Then $W_n(t)$  converges uniformly to $W(t)$ on any compact subset $K$ of $I\bigcap I'\bigcap\Omega_2$ in $\mathbb P$-$L^1$.

Moreover, set $$\Omega_1'=int\{t\in\mathbb R: \mathbb E \tilde Z_1(t)^\gamma<\infty \;\; \text{for some $\gamma >1$}\}.$$
If $K$ is a compact subset of $I\bigcap\Omega_1'$, then there exists a constant $p_K\in(1,2]$ such that $W_n(t)$ converges uniformly to $W(t)$ on $K$ in $\mathbb P$-$L^{p_K}$.
\end{thm}

It is clear that $\Omega_1'\subset\Omega_2$, but there is no evident relation between $I$ and $I'$. By calculating the derivative of  $\mathbb E\left[\frac{m_0(\gamma t)}{m_0(t)^\gamma}\right]$ with respect to $\gamma$ and letting $\gamma =1$, we can see that $I\bigcap \Omega_1'\subset I'\bigcap \Omega_1'$.

\subsection{Moderate deviation}
Finally we state a moderate deviation principle about the counting measures $Z_n$. Recently, Huang \& Liu showed the large deviation principle (\cite{huang22} , Theorem 3.2) and central limit theorem (\cite{huang22} , Theorem 7.1) about $Z_n$, which reflect the asymptotic properties of normalized measure $\frac{ Z_n(n\cdot)}{Z_n(\mathbb R)}$ and $\frac{Z_n(b_n\cdot)}{Z_n(\mathbb R)}$ (with some $b_n$ satisfies that $b_n/\sqrt{n}$ goes to a positive limit) . We want to establish the corresponding moderate deviation principle.

Let $(a_n)$ be a sequence of positive numbers satisfying
\begin{equation}
\frac{a_n}{n}\rightarrow0\qquad \text{and}\qquad \frac{a_n}{\sqrt{n}}\rightarrow\infty.
\end{equation}
We are interested in the asymptotic properties of normalized measure $\frac{ Z_n(a_n \cdot)}{Z_n(\mathbb R)}$.

\begin{thm}[Moderate deviation principle]\label{MDP}
Write $\pi_{0}=m_{0}(0)$. Assume that
$\|\frac{1}{\pi_0}\mathbb{E}_\xi\sum\limits_{i=1}^Ne^{\delta |L_i|}\|_\infty:=esssup\frac{1}{\pi_0}\mathbb{E}_\xi\sum\limits_{i=1}^Ne^{\delta |L_i|}<\infty$ for some
$\delta>0$
 and $\mathbb E_\xi \sum\limits_{i=1}^NL_i=0$ a.s.. If $0\in \Omega_1$, 
then  the sequence of finite measures
$A \mapsto  {Z_n(a_nA)}$ satisfies a principle of moderate deviation with
rate function $\frac{x^2}{2\sigma^2} $: for each measurable subset $A$ of
$\mathbb{R}$,
\begin{eqnarray*}\label{LDP1}
   - \frac{1}{2\sigma^2}\inf_{x\in A^o} x^2
   &\leq& \liminf_{n\rightarrow \infty}
             \frac{n}{a_n^2}  \log \frac{Z_n (a_nA)}{Z_n(\mathbb R)} \\
  & \leq &  \limsup_{n\rightarrow \infty}
              \frac{n}{a_n^2}  \log \frac{Z_n (a_nA)}{Z_n(\mathbb R)}
   \leq - \frac{1}{2\sigma^2}\inf_{x\in \bar A} x^2 \\
\end{eqnarray*}
a.s. on $\{Z_n(\mathbb R)\rightarrow\infty\}$, where  $\sigma^2=\mathbb E\left[\frac{1}{\pi_0}\sum\limits_{i=1}^NL_i^2\right]$,  $A^o$ denotes the interior of $A$, and $\bar A$ its closure.
\end{thm}

\bigskip

The rest part of the paper is arranged as follows. We first study the $L^p$ convergence and its exponential rate of the martingale $W_n$ in Section \ref{LPBRWS2} under quenched law and in Section \ref{LPBRWS3} under annealed law. Then we prove the uniform convergence of the martingale $W_n(t)$ in Section \ref{LPBRWS4}. Finally, in Section \ref{LPBRWS5}, we consider moderate deviations related to the counting measures $Z_n$.

\section{Quenched $L^p$ convergence; proof of Theorem \ref{LPBRW1.2}}\label{LPBRWS2}
In this section, we shall study the $L^p$ convergence of $W_n$ under quenched law $\mathbb{P}_\xi$ and its exponential rate.  To prove the results about the quenched convergence, we need the following lemma.

\begin{lem}\label{LPBRWL2.1}(\cite{huang14}, Lemma 3.1)
Let $(\alpha_n,\beta_n)_{n\geq0}$ be a stationary and ergodic sequence of non-negative random variables. If $\mathbb{E}\log\alpha_0<0$ and  $\mathbb{E}\log^+\beta_0<\infty$, then
\begin{equation}\label{CRE2.2.1}
\sum_{n=0}^{\infty}\alpha_0\cdots\alpha_{n-1}\beta_n<\infty\quad a.s..
\end{equation}
Conversely, if $\mathbb{E}|\log\beta_0|<\infty$, then (\ref{CRE2.2.1}) implies that $\mathbb{E}\log\alpha_0\leq0$.
\end{lem}

Recall that $W_n=W_n(1)$. To estimate the exponential rate of $W_n$, we consider the series introduced by Alsmeyer et al.  \cite{IK}:
\begin{equation}
A=A(\rho)=\sum_{n=0}^{\infty}\rho^n(W-W_n)\quad(\rho>1),
\end{equation}
\begin{equation}
\hat{A}_n=\hat{A}_n(\rho)=\sum_{k=0}^{n}\rho^{k}(W_{k+1}-W_k) \quad(\rho\geq1),
\end{equation}
\begin{equation}\hat{A}=\hat{A}(\rho)=\sum_{n=0}^{\infty}\rho^{n}(W_{n+1}-W_n)=\lim_{n\rightarrow\infty}\hat{A}_n \quad(\rho\geq1).
\end{equation}
According to (\cite{IK}, Lemma 3.1), with the same $\rho>1$,
 $A$ and $\hat A$ have the same convergence in the sense a.s. and in $L^p$ under $\mathbb{P}_\xi$ or $\mathbb{P}$. Since  $W_n$ is a martingale  under both laws $\mathbb{P}_\xi$ and $\mathbb{P}$, the same is true for $\hat A_n$ (but with respect to the filtration $\mathcal{F}_{n+1}$). In particular, if $\rho=1$, one can see that $\hat A_{n}=W_{n+1}-1$. Therefore we can study the convergence of $\hat A$ by  Doob's convergence theorems for martingales, which means that we should show a uniform upper bound for the $p$-th moment of $\hat A_n$ under $\mathbb{P}_\xi$ for quenched case and under $\mathbb{P}$ for annealed case. To this end, we will use  Bukholder's inequality as the basic tool. We mention that our approaches  are very similar to Huang \& Liu \cite{huang14} and Alsmeyer et al.  \cite{IK}, but the method of measure change for the annealed case would be heuristic.

\begin{lem}[Burkholder's inequality, see e.g. \cite{chow}]\label{CRL1.3.1}
Let $\{S_n\}$ be a $L^1$ martingale with $S_0=0$. Let $Q_n=(\sum\limits_{k=1}^{n}(S_k-S_{k-1})^2)^{1/2}$ and
$Q=(\sum\limits_{n=1}^{\infty}(S_n-S_{n-1})^2)^{1/2}$. Then $\forall p>1$,
\[c_p\parallel Q_n\parallel_p\leq\parallel S_n\parallel_p\leq C_p\parallel Q_n\parallel_p,\]
\[c_p\parallel Q\parallel_p\leq\sup_n\parallel S_n\parallel_p\leq C_p\parallel Q\parallel_p,\]
where $c_p=(p-1)/18p^{3/2}$, $C_p=18p^{3/2}/(p-1)^{1/2}$.
\end{lem}

Applying Burkholder's inequality, we can obtain the  moment results of $\hat A_n$ for  $1<p\leq 2$.

\begin{pr}[Quenched moments of $\hat A_n$: case $1<p\leq 2$]\label{LPBRW2.1}Assume (\ref{LPH1}).
Let $1<p\leq2$ and $\rho\geq1$.  If $\mathbb{E}\log\mathbb{E}_\xi W_1^r<\infty$ and $ \rho<\exp(-\frac{1}{r}\mathbb{E}\log m_0(r))$
for some $r\in[p,2]$, then  $\sup_n\mathbb E_\xi |\hat A_n|^p<\infty$ a.s.
\end{pr}

\begin{proof}Notice that
$$W_{n+1}-W_n=\sum_{u\in\mathbb T_n}\tilde X_u(W_{1,u}-1),$$
where we write $\tilde X_u=\tilde X_u(1)$ for short, and under quenched law $\mathbb P_\xi$, $\{W_{k,u}(t)\}_{|u|=n}$ are i.i.d. and independent of $\mathcal F_n$ with common distribution determined by $\mathbb P_\xi(W_{k,u}(t)\in \cdot)=\mathbb P_{T^n\xi}(W_k(t)\in\cdot)$. The notation $T$ represents the shift operator: $T^n\xi=(\xi_n,\xi_{n+1},\cdots)$ if $\xi=(\xi_0,\xi_1,\cdots)$. Applying Burkholder's inequality to $W_{n+1}-W_n$, and noticing the concavity of
$x^{p/2}$, $x^{r/2}$ and $x^{p/r}$, we have
\begin{eqnarray}\label{que1}
\mathbb{E}_\xi|W_{n+1}-W_n|^p&\leq& C\mathbb{E}_\xi\left(\sum_{u\in\mathbb T_n}\tilde X_u^2(W_1(u)-1)^2\right)^{p/2}\nonumber\\
&\leq&C\left(\mathbb{E}_\xi\sum_{u\in\mathbb T_n}\tilde X_u^r|W_1(u)-1|^r\right)^{p/r}\nonumber\\
&=&C P_n(r)^{p/r}\left(\mathbb{E}_{T^n\xi}|W_1-1|^r\right)^{p/r},
 \end{eqnarray}
where $C$ is positive constant, and in general, it does not stand for the same constant throughout.
Noticing (\ref{que1}), and applying again Burkholder's inequality to $\hat A_n$ gives
\begin{eqnarray*}
\sup_n\mathbb{E}_\xi|\hat{A}_n|^p
&\leq&C\mathbb{E}_\xi\left(\sum_{n=0}^{\infty}\rho^{2n}(W_{n+1}-W_n)^2\right)^{p/2}\\
&\leq&C\sum_{n=0}^{\infty}\rho^{pn}\mathbb{E}_\xi|W_{n+1}-W_n|^p\\
&\leq&C\sum_{n=0}^{\infty}\rho^{pn} P_n(r)^{p/r}(\mathbb{E}_{T^n\xi}|W_1-1|^r)^{p/r}.
 \end{eqnarray*}
Since $\mathbb{E}\log \mathbb{E}_\xi W_1^r<\infty$ and $\log \rho+ \mathbb{E}\log m_0(r)/r<0$, by Lemma \ref{LPBRWL2.1}, the series $\sum_n\rho^{pn} P_n(r)^{p/r}(\mathbb{E}_{T^n\xi}|W_1-1|^r)^{p/r}$ converges a.s., which leads to $\sup_n\mathbb{E}_\xi |\hat A_n|^p<\infty$ a.s..
\end{proof}


For $p>2$, we also have results for the quenched moments of $\hat A_n$.

\begin{pr}[Quenched moments of $\hat A_n$: case $p>2$]\label{CMDPP1}Assume (\ref{LPH1}). Let $p\geq2$ and $\rho\geq1$. 
\begin{itemize}
\item[(a)] If $2<p\leq t_+$, $\mathbb{E}\log\mathbb{E}_\xi W_1^p<\infty$ and $ \rho<\exp(-\frac{1}{2}\mathbb{E}\log m_0(2))$, then $\sup_n\mathbb E_\xi |\hat A_n|^p<\infty$ a.s..
\item[(b)] If  $\mathbb E|\log\mathbb E_\xi |W_1-1|^2|<\infty$, then
$\sup_n\mathbb E_\xi |\hat A_n|^p<\infty$ a.s. implies that $ \rho\leq\exp(-\frac{1}{2}\mathbb{E}\log m_0(2))$.
\end{itemize}
\end{pr}

The proof of Proposition \ref{CMDPP1} is based on the following result about the moment  of $W_n$.

\begin{lem}\label{LPL1}
Fix $t>0$. If $1<p\leq \frac{t_+}{t}$ and $\mathbb E \log \mathbb E_\xi W_1(t)^p<\infty$, then
\begin{equation}\label{LPE1}
\sup_n\mathbb E_\xi W_n(t)^p<\infty\qquad a.s..
\end{equation}
\end{lem}

\begin{proof}
Assume that $p\in(2^b, 2^{b+1}]$ for some integer $b\geq0$. We will prove (\ref{LPE1}) by induction on $b$. Firstly, for $b=0$, we have
$1<p\leq 2$. Recall that $\Lambda(t)=\mathbb E\log m_0(t)$. For $t>0$ fixed, set $h_t(x)=\frac{1}{x}\Lambda(tx)$, whose derivative is
$h_t'(t)=\frac{1}{x^2}(tx\Lambda'(tx)-\Lambda(tx))$. Since $h_t'(x)<0$ on $(\frac{t_-}{t}, \frac{t_+}{t})$, the function $h_t(x)$ is strictly decreasing on  $[\frac{t_-}{t}, \frac{t_+}{t}]$. Thus
$$\frac{1}{p}\mathbb E\log\frac{m_0(pt)}{m_0(t)^p}=h_t(p)-h_t(1)<0.$$
Noticing $\mathbb E \log \mathbb E_\xi W_1(t)^p<\infty$, we obtain (\ref{LPE1}) by applying Proposition \ref{LPBRW2.1} with $\rho=1$.

 Now suppose the conclusion holds for $p\in(2^{b}, 2^{b+1}]$. For  $p\in(2^{b+1}, 2^{b+2}]$, we have $p/2\in(2^b,2^{b+1}]$. Observe that
$$\mathbb E_\xi W_1(2t)^{p/2}=\mathbb E_\xi\left(\sum_{u\in\mathbb T_1}\frac{e^{2tS_u}}{m_0(2t)}\right)^{p/2}\leq \mathbb E_\xi \left[ \frac{\left(\sum_{u\in\mathbb T_1}e^{tS_u}\right)^p}{m_0(2t)^{p/2}}\right]=\mathbb E_\xi W_1(t)^p\frac{m_0(t)^p}{m_0(2t)^{p/2}}.$$
It follows that $\mathbb E\log \mathbb E_\xi W_1(t)^p<\infty$ implies $\mathbb E\log \mathbb E_\xi W_1(2t)^{p/2}<\infty$. As $1<\frac{p}{2}\leq\frac{t_+}{2t}$, by induction, we have \begin{equation}\label{Lp4}
\sup_n \mathbb E_\xi W_n(2t)^{p/2}<\infty\qquad a.s..
\end{equation}
Notice that by Burkholder's inequality and Minkowski's inequality,
$$\sup_n\mathbb E_\xi |W_n(t)-1|^p \leq C\left(\sum_{n=0}^\infty\left(\mathbb E_\xi|W_{n+1}(t)-W_n(t)|^p\right)^{2/p}\right)^{p/2}.$$
Using Burkholder's inequality and Jensen's inequality, we have
\begin{eqnarray}\label{Lp2}
\mathbb E_\xi |W_{n+1}(t)-W_n(t)|^p&\leq& C\mathbb E_\xi\left(\sum_{u\in\mathbb T_n}\tilde X_u^2(t)(W_{1,u}(t)-1)^2\right)^{p/2}\nonumber\\
&\leq&C\mathbb{E}_\xi\left(\sum_{u\in\mathbb T_n}\tilde X_u^2(t)\right)^{p/2-1}\sum_{u\in\mathbb T_n}\tilde X_u^2(t)|W_{1,u}(t)-1|^p\nonumber\\
&=&C\mathbb{E}_\xi\left(\sum_{u\in\mathbb T_n}\tilde X_u^2(t)\right)^{p/2}\mathbb E_{ T^n\xi}|W_1(t)-1|^p\nonumber\\
&=&C \frac{P_n(2t)^{p/2}}{P_n(t)^p}\mathbb E_\xi W_n(2t)^{p/2}\mathbb E_{ T^n\xi}|W_1(t)-1|^p.
\end{eqnarray}
Noticing (\ref{Lp4}), to get (\ref{LPE1}), it suffices to show the convergence of the series
\begin{equation}\label{Lp3}
\sum_n \frac{P_n(2t)}{P_n(t)^2}\left(\mathbb E_{ T^n\xi}|W_1(t)-1|^p\right)^{2/p}.
\end{equation}
Since $1<2<p\leq \frac{t_+}{t}$, we have
$$\mathbb E\log m_0(t)=h_t(1)>h_t(2)=\frac{1}{2}\mathbb E\log m_0(2t).$$
By Lemma \ref{LPBRWL2.1}, the series (\ref{Lp3})  converges a.s.. The proof is complete.
\end{proof}

\begin{proof}[Proof of Proposition \ref{CMDPP1}]
We first consider the assertion (a).
The conclusion for $\rho=1$ is contained in Lemma \ref{LPL1}. For $\rho>1$, similarly to the proof of Lemma \ref{LPL1}, applying Burkholder's inequality to $\hat A_n$ and noticing (\ref{Lp2}), the conclusion follows  by the convergence of the series
\begin{equation}\label{LPPE22}
\sum_n \rho^{2n}P_n(2)(\mathbb E_\xi W_n(2)^{p/2})^{2/p}\left(\mathbb E_{ T^n\xi}|W_1-1|^p\right)^{2/p}.
\end{equation}
Since  $\mathbb E\log \mathbb E_\xi W_1^p<\infty$ (so that $\mathbb E\log \mathbb E_\xi W_1(2)^{p/2}<\infty$) and $1<\frac{p}{2}\leq\frac{t_+}{2}$, Lemma \ref{LPL1} gives $\sup_n\mathbb E_\xi W_n(2)^{p/2}<\infty$ a.s.. By Lemma \ref{LPBRWL2.1}, the series (\ref{LPPE22}) converges a.s. if $ \rho<\exp(-\frac{1}{2}\mathbb{E}\log m_0(2))$.

We next consider the assertion (b). By Burkholder's inequality, 
\begin{eqnarray*}
\sup_n\mathbb{E}_\xi|\hat{A}_n|^p&\geq&C\sum_{n=0}^\infty \rho^{pn}\mathbb{E}_\xi|W_{n+1}-W_n|^p\\
&\geq&C\sum_{n=0}^\infty \rho^{pn}\mathbb{E}_\xi\left(\sum_{u\in\mathbb T_n}\tilde X_u^2(W_{1,u}-1)^2\right)^{p/2}\\
&\geq&C\sum_{n=0}^\infty \rho^{pn}\left(\mathbb{E}_\xi\sum_{u\in\mathbb T_n}\tilde X_u^2(W_{1,u}-1)^2\right)^{p/2}\\
&=&C\sum_{n=0}^\infty  \rho^{pn}P_n(2)^{p/2}(\mathbb{E}_{ T^n\xi}|W_1-1|^2)^{p/2}.
\end{eqnarray*}
Since $\mathbb E|\log\mathbb E_\xi |W_1-1|^2|<\infty$, we deduce  $\rho\leq\exp(-\frac{1}{2}\mathbb{E}\log m_0(2))$ by Lemma \ref{LPBRWL2.1}.
\end{proof}

Now we give the proof of Theorem \ref{LPBRW1.2}.
\begin{proof}[Proof of Theorem \ref{LPBRW1.2}]
For the assertion (a),
by Proposition \ref{LPBRW2.1}, we have $\sup_n\mathbb{E}_\xi |\hat A_n|^p<\infty$ a.s., which implies that $\rho^n(W_n-W)\rightarrow0$  in $\mathbb P_\xi$-$L^p$ for almost all $\xi$.
For the assertion (b), if $\rho<\rho_c$,  by Proposition \ref{CMDPP1}(a), we have
$\sup_n\mathbb E_\xi|\hat A_n|^p<\infty$ a.s., so that $\rho^n(W-W_n)\rightarrow0$ in $\mathbb{P}_\xi$-$L^p$ for almost all $\xi$.
If $\rho>\rho_c$ and $ \mathbb E\log^-\mathbb E_\xi |W_1-1|^2<\infty$,  we assume that $\rho^n(\mathbb{E}|W-W_n|^p)^{1/p}\rightarrow0$ a.s.. Denote $D=\{\xi: \rho^n(\mathbb{E}|W-W_n|^p)^{1/p}\rightarrow0\}$. Following similar argument in (\cite{huang14}, proof of Theorem 1.2), we can see that $\mathbb P(D)=0$ or $1$. If $\mathbb P(D)=1$, the sequence $\rho^n(\mathbb{E}_\xi|W-W_n|^p)^{1/p}$ is bounded for almost all $\xi$.  Denote this bound by $M(\xi)$. For any $1<\rho_1<\rho$,  the series $\sum_n\rho_1^n(W-W_n)$ converges in $\mathbb{P}_\xi$-$L^p$  since
$$\left(\mathbb{E}_\xi\left|\sum_n\rho_1^n(W-W_n)\right|^p\right)^{1/p}\leq \sum_n\rho_1^n(\mathbb{E}_\xi|W-W_n|^p)^{1/p}\leq M(\xi)\sum_n\left(\frac{\rho_1}{\rho}\right)^n<\infty \quad a.s..$$
Thus $A(\rho_1)$ and $\hat A(\rho_1)$ converge in $\mathbb P_\xi$-$L^p$, so that $\sup_n\mathbb E_\xi|\hat A_n(\rho_1)|^p<\infty$ a.s.. By Proposition \ref{CMDPP1}(b), we have $\rho_1\leq\rho_c$. Letting $\rho_1\uparrow\rho$ yields $\rho\leq\rho_c$. This contradicts the fact that $\rho>\rho_c$. Thus $\mathbb P(D)=0$, i.e., $\rho^n(\mathbb{E}_\xi|W-W_n|^p)^{1/p}\nrightarrow0$ for almost all $\xi$.
\end{proof}


\section{Annealed $L^p$ convergence}\label{LPBRWS3}
\subsection{Change of measure}\label{CMDPS3.1}
Inspired by the idea of the classic measure change (see for example Lyons \cite{lyons}, Biggins \& Kyprianou  \cite{biggins04}, Hu \& Shi  \cite{hu}), we introduce a new probability measure as follows.

When the environment $\xi$ is given, for $t\in\mathbb{R}$ fixed, define a new probability $\mathbb{Q}_\xi=\mathbb{Q}_\xi^{(t)}$  such that for any $n\geq1$,
\begin{equation}
\mathbb{Q}_\xi|_{\mathcal{F}_n}=W_n(t)\mathbb{P}_\xi|_{\mathcal{F}_n}.
\end{equation}
The existence of $\mathbb{Q}_\xi$ is ensured by Kolmogorov's extension theorem. In fact, under $\mathbb{Q}_\xi$, the tree $\mathbb{T}$ is a so-called size-biased weighted tree (see for example Kuhlbusch (2004, \cite{ku}) for the construction of a size-biased tree).

Fix $n\geq1$. Let $\omega_n^n=\omega_n^n(t)$ be a random variable taking values in $\mathbb{T}_n$ such that for any $u\in\mathbb T_n$,
\begin{equation}
\mathbb{Q}_\xi(\omega_n^n=u|\mathcal{F}_\infty)=\frac{\tilde X_u(t)}{W_n(t)},
\end{equation}
where $\mathcal{F}_{\infty}=\sigma(\mathcal{F}_n,n\geq0)$. Denote $\omega_k^n=\omega_n^n|k$ for $k=0,1,\cdots, n$. So $(\omega_0^n, \omega_1^n, \cdots, \omega_n^n)$ is the vertices visited by the shortest path in $\mathbb{T}$ connecting the root $\omega_0^n=\emptyset$ with $\omega_n^n$.

\begin{lem}\label{LPBRWL3.1}
Fix $t\in\mathbb{R}$ and $n\geq1$. For all nonnegative Borel functions $h$ and $g$ (defined on $\mathbb{R}$ or  $\mathbb{R}^2$), we have for each $k=1,2,\cdots, n$,
\begin{eqnarray}
&&\mathbb{E}_{\mathbb{Q}_\xi}h\left(\tilde X_{\omega_k^n}(t), \sum_{\substack{ v\in\mathbb T_k\\ v\neq\omega_k^n}}\tilde X_v(t)W_{n-k,v}(t)\right)g\left(W_{n-k,\omega_k^n}(t)\right)\nonumber\\
&=&\mathbb{E}_\xi\sum_{v\in\mathbb T_k} \tilde X_u(t)h\left(\tilde X_u(t), \sum_{\substack{v\in\mathbb T_k\\ v\neq u}}\tilde X_v(t)W_{n-k,v}(t)\right)\mathbb{E}_{T^k\xi}W_{n-k}(t)g\left(W_{n-k}(t) \right).
\label{LPBRWE3.1}
\end{eqnarray}
\end{lem}

\begin{proof}[Proof]
By the definition of $\omega^n_k$ and $\mathbb{Q}_\xi$, we can calculate that
\begin{eqnarray*}
&&\mathbb{E}_{\mathbb{Q}_\xi}h\left(\tilde X _{\omega_k^n}(t), \sum_{\substack{ v\in \mathbb T_k \\ v\neq\omega_k^n}}\tilde X_v(t)W_{n-k,v}(t)\right)g\left(W_{n-k,\omega_k^n}(t)\right)\\
&=&\mathbb{E}_{\mathbb{Q}_\xi}\sum_{u\in\mathbb T_n}\mathbf{1}_{\{u=\omega_n^n\}}h\left(\tilde X_{u|k}(t), \sum_{\substack{v\in\mathbb T_k\\ v\neq u|k } }\tilde X_v(t)W_{n-k,v}(t)\right)g\left(W_{n-k,u|k}(t)\right)\\
&=&\mathbb{E}_{\mathbb{Q}_\xi}\sum_{u\in\mathbb T_n} \mathbb{Q}_\xi(\omega_n^n=u|\mathcal{F}_\infty)h\left(\tilde X_{u|k}(t), \sum_{\substack{ v\in\mathbb T_k\\v\neq u|k }}\tilde X_v(t)W_{n-k,v}(t)\right)g\left(W_{n-k,u|k}(t)\right)\\
 &=&\mathbb{E}_{\mathbb{Q}_\xi}\sum_{u\in\mathbb T_n} \frac{\tilde X_u(t)}{W_n(t)}h\left(\tilde X_{u|k}(t), \sum_{\substack{v\in\mathbb T_k\\ v\neq u|k }}\tilde X_v(t)W_{n-k,v}(t)\right)g\left(W_{n-k,u|k}(t)\right)\\
&=&\mathbb{E}_{\xi}\sum_{u\in\mathbb T_n}\tilde X_u(t)h\left(\tilde X_{u|k}(t), \sum_{\substack{v\in\mathbb T_k\\v\neq u|k }}\tilde X_v(t)W_{n-k,v}(t)\right)g\left(W_{n-k,u|k}(t)\right)\\
&=&\mathbb{E}_{\xi}\sum_{u\in\mathbb T_k}\tilde X_u(t)W_{n-k,u}(t)h\left(\tilde X_{u}(t), \sum_{\substack{v\in\mathbb T_k\\ v\neq u }}\tilde X_v(t)W_{n-k,v}(t)\right)g\left(W_{n-k,u}(t)\right)\\
&=&\mathbb{E}_\xi\sum_{u\in\mathbb T_k}\tilde X_u(t)h\left(\tilde X_{u}(t), \sum_{\substack{v\in\mathbb T_k\\ v\neq u}}\tilde X_v(t)W_{n-k,v}(t)\right)\mathbb{E}_{T^k\xi}W_{n-k}(t)g\left(W_{n-k}(t) \right).
\end{eqnarray*}
\end{proof}

\begin{re}\emph{In particular, taking $h=1$ or $g=1$ gives
\begin{equation}\label{LPBRWE3.2}
\mathbb{E}_{\mathbb{Q}_\xi}g\left(W_{n-k,\omega_k^n}(t)\right)=\mathbb{E}_{T^k\xi}W_{n-k}(t)g\left(W_{n-k}(t) \right),
\end{equation}
\begin{equation}\label{LPBRWE3.3}
\mathbb{E}_{\mathbb{Q}_\xi}h\left(\tilde X_{\omega_k^n}(t), \sum_{\substack{v\in\mathbb T_k\\ v\neq\omega_k^n}}\tilde X_v(t)W_{n-k,v}(t)\right)=\mathbb{E}_\xi\sum_{v\in\mathbb T_k}\tilde X_u(t)h\left(\tilde X_{u}(t), \sum_{\substack{v\in\mathbb T_k\\v\neq u}}\tilde X_v(t)W_{n-k,v}(t)\right).
\end{equation}
Combing (\ref{LPBRWE3.2}), (\ref{LPBRWE3.3}) with (\ref{LPBRWE3.1}), we have
\begin{eqnarray*}
&&\mathbb{E}_{\mathbb{Q}_\xi}h\left(\tilde X_{\omega_k^n}(t), \sum_{\substack{v\in\mathbb T_k\\ v\neq\omega_k^n}}\tilde X_v(t)W_{n-k,v}(t)\right)g\left(W_{n-k,\omega_k^n}(t)\right)\\
&=&\mathbb{E}_{\mathbb{Q}_\xi}h\left(\tilde X_{\omega_k^n}(t), \sum_{\substack{v\in\mathbb T_k\\ v\neq\omega_k^n}}\tilde X_v(t)W_{n-k,v}(t)\right)\mathbb{E}_{\mathbb{Q}_\xi}g\left(W_{n-k,\omega_k^n}(t)\right),
\end{eqnarray*}
which means that the random vector  $\left(\tilde X_{\omega_k^n}(t), \sum\limits_{\substack{v\in\mathbb T_k\\ v\neq\omega_k^n}}\tilde X_v(t)W_{n-k,v}(t)\right)$  is independent of $W_{n-k,\omega_k^n}(t)$ under $\mathbb{Q}_\xi$. Moreover, for all nonnegative Borel functions $f$ defined on $\mathbb{R}$, by taking $h(x,y)=f(x)$ or $f(y)$, we obtain
$$\mathbb{E}_{\mathbb{Q}_\xi}f\left(\tilde X _{\omega_k^n}(t)\right)=\mathbb{E}_\xi\sum_{u\in\mathbb T_k}\tilde X_u(t)f\left(\tilde X_{u}(t)\right),$$
$$\mathbb{E}_{\mathbb{Q}_\xi}f\left( \sum_{\substack{v\in\mathbb T_k\\ v\neq\omega_k^n}}\tilde X_v(t)W_{n-k,v}(t)\right)=\mathbb{E}_\xi\sum_{u\in\mathbb T_k}\tilde X_u(t)f\left(\sum_{\substack{v\in\mathbb T_k \\v\neq u}}\tilde X_v(t)W_{n-k,v}(t) \right).$$
These above assertions generalize the results of Liu (\cite{liu2000}, Lemma 4.1) on generalized multiplicative cascades.}
\end{re}


\subsection{Auxiliary  results}
In this section, We shall obtain some auxiliary results for the study of the annealed $L^p$ convergence rate of $W_n$. Let's consider the i.i.d. environment, where $(\xi_n)$ are i.i.d.
Denote
$$U_n^{(t)}(s,r)=\mathbb{E}P_n(t)^sW_n(t)^r\quad(s,t\in\mathbb{R},r>1).$$
We will show two lemmas about $U_n^{(t)}(s,r)$: the first one is a recursive inequality; the second one gives a upper estimation. Particularly, the results for $t=0$ were already shown in \cite{huang14}.
\begin{lem}\label{LPBRWL4.1}
Let $r>2$. Then
\begin{eqnarray}\label{recu}
U_n^{(t)}(s,r)^{\frac{1}{r-1}}\leq \left[\mathbb{E}m_0(t)^{s-r}m_0(tr)\right]^{\frac{1}{r-1}}U_{n-1}^{(t)}(s,r)^{\frac{1}{r-1}}+
\left[\mathbb{E}m_0(t)^sW_1(t)^r\right]^{\frac{1}{r-1}}U_{n-1}^{(t)}(s,r-1)^{\frac{1}{r-1}}.
\end{eqnarray}
\end{lem}

\begin{proof}[Proof]
Fix $t\in\mathbb R$. Given $\xi$, we consider the probability $\mathbb{Q}_\xi$ defined in Section \ref{CMDPS3.1}. Notice that
$$W_n(t)=\sum_{u\in\mathbb T_n}\tilde X_u(t)=\sum_{u\in\mathbb T_1}\tilde X_u^{(t)}W_{n-1,u}(t).$$
We have
$$\mathbb{E}_\xi W_n(t)^r=\mathbb{E} _{\mathbb{Q}_\xi}W_{n-1}(t)^{r-1}=\mathbb{E} _{\mathbb{Q}_\xi}\left(\sum_{u\in\mathbb T_1}\tilde X_u(t)W_{n-1,u}(t)\right)^{r-1}.
$$
Thus
\begin{eqnarray*}
U_n^{(t)}(s,r)&=&\mathbb{E}P_n(t)^s\mathbb{E}_\xi W_n(t)^r\\
&=&\mathbb{E}P_n(t)^s\mathbb{E} _{\mathbb{Q}_\xi}\left(\sum_{u\in\mathbb T_1}\tilde X_u(t)W_{n-1,u}(t)\right)^{r-1}\\
&=&\mathbb{E} _{\mathbb{Q}}\left(P_n(t)^{\frac{s}{r-1}}\sum_{u\in\mathbb T_1}\tilde X_u(t)W_{n-1,u}(t)\right)^{r-1},
\end{eqnarray*}
where the probability  $\mathbb{Q}$ is defined as $\mathbb{Q}(B)=\mathbb{E}\mathbb{Q}_\xi(B)$, for any measurable set $B$.

Fix $n\geq1$. The set $\mathbb T_1$ can be divided into two parts: $\{\omega_1^n\}$ and the set of his brothers $\{u\in \mathbb{T}_1: u\neq \omega_1^n\}$.
We therefore   have
\begin{eqnarray*}
\sum_{u\in\mathbb T_1}\tilde X_u(t)W_{n-1,u}(t)&=&\tilde X_{\omega_1^n}(t)W_{n-1,\omega_1^n}(t)
+\sum_{\substack{u\in\mathbb T_1\\ u\neq \omega_1^n}}\tilde X_u(t)W_{n-1,u}(t)\\
&=&:M_n+Q_n.
\end{eqnarray*}
Hence
$$U_n^{(t)}(s,r)= \mathbb{E} _{\mathbb{Q}}\left[P_n(t)^{\frac{s}{r-1}}M_n+P_n(t)^{\frac{s}{r-1}}Q_n\right]^{r-1}.$$
By Minkowski's inequality,
\begin{equation}\label{m}
U_n^{(t)}(s,r)^{\frac{1}{r-1}}\leq\left[\mathbb{E} _{\mathbb{Q}}P_n(t)^sM_n^{r-1}\right]^{\frac{1}{r-1}}+
\left[\mathbb{E} _{\mathbb{Q}}P_n(t)^sQ_n^{r-1}\right]^{\frac{1}{r-1}}.
\end{equation}
By Lemma \ref{LPBRWL3.1}, we can calculate a.s.,
\begin{eqnarray*}
\mathbb{E} _{\mathbb{Q}_\xi}M_n^{r-1}&=&\mathbb{E} _{\mathbb{Q}_\xi}\left[\tilde X_{\omega_1^n}(t)W_{n-1,\omega_1^n}(t)\right]^{r-1}\\
&=&\mathbb{E}_\xi\sum_{u\in\mathbb T_1}\tilde X_u(t)^r\mathbb{E}_{T\xi}W_{n-1}(t)^{r}\\
&=&m_0(tr)m_0(t)^{-r}\mathbb{E}_{T\xi}W_{n-1}(t)^{r}.
\end{eqnarray*}
Therefore, by the independency of $(\xi_n)$,
\begin{eqnarray}\label{M}
\mathbb{E} _{\mathbb{Q}}P_n(t)^sM_n^{r-1}&=&\mathbb{E}P_n(t)^s\mathbb{E} _{\mathbb{Q}_\xi}M_n^{r-1}\nonumber\\
&=&\mathbb{E}P_n(t)^sm_0(tr)m_0(t)^{-r}\mathbb{E}_{T\xi}W_{n-1}(t)^{r}\nonumber\\
&=&\mathbb{E}m_0(t)^{s-r}m_0(tr)\mathbb{E}P_{n-1}(t)^sW_{n-1}(t)^{r}\nonumber\\
&=&\mathbb{E}m_0(t)^{s-r}m_0(tr)U_{n-1}^{(t)}(s,r).
\end{eqnarray}
Similarly, again by Lemma \ref{LPBRWL3.1}, we have a.s.
\begin{eqnarray*}
\mathbb{E} _{\mathbb{Q}_\xi}Q_n^{r-1}&=&\mathbb{E} _{\mathbb{Q}_\xi}\left(\sum_{\substack{u\in\mathbb T_1\\ u\neq\omega_1^n}}\tilde X_u(t)W_{n-1,u}(t)\right)^{r-1}\\
&=&\mathbb{E}_\xi\sum_{u\in\mathbb T_1}\tilde X_u(t)\left(\sum_{\substack{v\in\mathbb T_1\\v\neq u}}\tilde X_v(t)W_{n-1,v}(t)\right)^{r-1}\\
&\leq&\mathbb{E}_\xi\sum_{u\in\mathbb T_1}\tilde X_u(t)\left(\sum_{\substack{v\in\mathbb T_1\\v\neq u}}\tilde X_v(t)\right)^{r-2}\sum_{\substack{v\in\mathbb T_1\\v\neq u}}{\tilde X_v(t)}W_{n-1,v}(t)^{r-1}\\
&=&\mathbb{E}_\xi\sum_{u\in\mathbb T_1}\tilde X_u(t)\left(\sum_{\substack{v\in\mathbb T_1\\v\neq u}}\tilde X_v(t)\right)^{r-1}\mathbb{E}_{T\xi}
W_{n-1}(t)^{r-1}\\&\leq&\mathbb{E}_\xi\left(\sum_{u\in\mathbb T_1}\tilde X_u(t)\right)^{r}\mathbb{E}_{T\xi}W_{n-1}(t)^{r-1}\\
&=&\mathbb{E}_\xi W_1(t)^r \mathbb{E}_{T\xi}W_{n-1}(t)^{r-1}.\\
\end{eqnarray*}
Thus,
\begin{eqnarray}\label{Q}
\mathbb{E} _{\mathbb{Q}}P_n(t)^s Q_n^{r-1}&=&\mathbb{E}P_n(t)^s\mathbb{E} _{\mathbb{Q}_\xi}Q_n^{r-1}\nonumber\\
&\leq&\mathbb{E}P_n(t)^s \mathbb{E}_\xi W_1(t)^r \mathbb{E}_{T\xi}W_{n-1}(t)^{r-1}\nonumber\\
&=&\mathbb{E}m_0(t)^{s}W_1(t)^r U_{n-1}^{(t)}(s,r-1).
\end{eqnarray}
Combing (\ref{M}), (\ref{Q}) with (\ref{m}), we obtain (\ref{recu}).
\end{proof}

Following similar arguments of (\cite{huang14}, Lemma 4.4),  we obtain the following lemma which generalize (\cite{huang14}, Lemma 4.4)  to BRWRE.  The proof is omitted.

\begin{lem}\label{LPBRWL4.2}
If $\mathbb{E}m_0(t)^s<\infty$ and $\mathbb{E}m_0(t)^sW_1(t)^r<\infty$, then
\begin{itemize}
\item[] (i) for $r\in(1,2]$,
$$U_n^{(t)}(s,r)\leq C n\left[\max\left\{\mathbb{E}m_0(t)^{s-r}m_0(tr), \mathbb{E}m_0(t)^s\right\}\right]^n; $$
\item[](ii) for $r\in(b+1,b+2]$, where $b\geq1$ is an integer,
$$U_n^{(t)}(s,r)\leq C n^{br}\left[\max\left\{\left(\mathbb{E}m_0(t)^{s-r+i}m_0(t(r-i))\right)_{0\leq i\leq b, i\in \mathbb{N}}, \mathbb{E}m_0(t)^s\right\}\right]^n, $$
\end{itemize}
where $C$ is a general constant depending on $r,s$ and $t$.
\end{lem}

\begin{lem}\label{LPBRWL4.3}The function
$f(x):=\mathbb{E}m_0(t)^xm_0(\alpha+\beta x)$ ($t, \alpha$ and $\beta \in\mathbb{R}$ are fixed) is $\log$ convex.
\end{lem}

\begin{proof}[Proof]
For $\lambda\in(0,1)$, $\forall x_1,x_2$, using H\"older's inequality, we have
\begin{eqnarray*}
m_0(\alpha+\beta(\lambda x_1+(1-\lambda)x_2))\leq m_0(\alpha+\beta x_1)^\lambda m_0(\alpha+\beta x_2)^{1-\lambda},
\end{eqnarray*}
which means that $m_0(\alpha+\beta x)$ is log convex.
Noticing the inequality above and using H\"older's inequality again,  we get
\begin{eqnarray*}
 f(\lambda x_1+(1-\lambda )x_2)&=& \mathbb{E}m_0(t)^{\lambda x_1+(1-\lambda)x_2}m_0(\alpha+\beta(\lambda x_1+(1-\lambda)x_2))\\
&\leq& \mathbb{E}\left[m_0(t)^{x_1}m_0(\alpha+\beta x_1)\right]^{\lambda}\left[m_0(t)^{x_2}m_0(\alpha+\beta x_2)\right]^{1-\lambda}\\
&\leq& \left[\mathbb{E}m_0(t)^{x_1}m_0(\alpha+\beta x_1)\right]^{\lambda}\left[\mathbb{E}m_0(t)^{x_2}m_0(\alpha+\beta x_2)\right]^{1-\lambda}\\
&=& f(x_1)^\lambda f(x_2)^{1-\lambda},
\end{eqnarray*}
which confirms the log-convexity of $f$.
\end{proof}

\subsection{Proofs of Theorems \ref{LPBRW1.3} and \ref{LPBRW1.4}}
Recall the martingale $\hat A_n$ introduced in Section \ref{LPBRWS2}. Similar to the quenched case, we need to study the $p$-th ($p>1$) moment of $\hat A_n$ under annealed law $\mathbb P$. We distinguish two case: $1<p<2$ and $p\geq2$.

\begin{pr}[Annealed moments of $\hat A_n$: case $1<p<2$]\label{LPBRWT5.1}
Assume (\ref{LPH1}) and that $(\xi_n)$ are i.i.d..
Let $1<p<2$ and $\rho\geq1$. If
$ \mathbb{E}[\mathbb{E}_\xi W_1^r]^{p/r}<\infty$ and $\rho\left[\mathbb{E}m_0(r)^{p/r}\right]^{1/p}<1$
for some $r\in[p,2]$, then  $\sup_n\mathbb E|\hat A_n|^p<\infty$.
\end{pr}

\begin{proof}[Proof]
 Similar to the proof of Theorem \ref{LPBRW2.1}, appling Burkholder's inequality to $\hat A_n$ under annealed law $\mathbb P$ and noticing (\ref{que1}), we have
\begin{eqnarray*}
\sup_n\mathbb{E}|\hat{A}_n|^p&\leq& C \sum_{n=0}^{\infty}\rho^{pn}\mathbb{E}\left[P_n(r)\mathbb{E}_{T^n\xi}|W_1-1|^r\right]^{p/r}\\
&=&C\mathbb{E}(\mathbb{E}_\xi|W_1-1|^r)^{p/r}\sum_{n=0}^{\infty}\rho^{pn}[\mathbb{E}m_0(r)^{p/r}]^n.
\end{eqnarray*}
Thus $\sup_n\mathbb{E} | \hat{A}_n|^p<\infty $ if $\mathbb{E}[\mathbb{E}_\xi W_1^r]^{p/r}<\infty$ and $\rho(\mathbb{E}m_0(r)^{p/r})^{1/p}<1$.
\end{proof}

Now we consider the case where $p\geq2$. The proposition below gives a sufficient and necessary condition for the  existence of uniform $p$-th moment of $\hat A_n$ under annealed law $\mathbb{P}$.

\begin{pr}[Annealed moments of $\hat A_n$: case $p\geq2$]\label{LPBRWT5.2}
Assume  (\ref{LPH1}), (\ref{LPH2}) and that $(\xi_n)$ are i.i.d..
Let $p\geq2$ and $\rho\geq1$. Then  $\sup_n\mathbb E|\hat A_n|^p<\infty$ if and only if
 $\mathbb{E}W_1^p<\infty$ and $\rho\max\{[\mathbb{E}m_0(p)]^{1/p},[\mathbb{E}m_0(2)^{p/2}]^{1/p}\}<1$.
\end{pr}


\begin{proof}
(i) The necessity. Since $\sup_n\mathbb{E}| \hat{A}_n|^p<\infty $, we have $\mathbb{E}W_1^p<\infty$. Furthermore, by Burkholder's inequality, we can calculate for all $r\in[2, p]$,
\begin{eqnarray*}
\sup_n\mathbb{E}|\hat{A}_n|^p&\geq&
C\sum_{n=0}^\infty \rho^{pn}\mathbb{E}\left(\sum_{u\in\mathbb T_n}\tilde X_u^2(W_{1,u}-1)^2\right)^{p/2}\\
&\geq&C\sum_{n=0}^\infty \rho^{pn}\mathbb{E}\left(\mathbb{E}_\xi\sum_{u\in\mathbb T_n}\tilde X_u^r|W_{1,u}-1|^r\right)^{p/r}\\
&=&C\mathbb{E}(\mathbb{E}_\xi|W_1-1|^r)^{p/r}\sum_{n=0}^\infty  \rho^{pn}[\mathbb{E}m_0(r)^{p/r}]^n.
\end{eqnarray*}
Therefore the series $\sum_n  \rho^{pn}[\mathbb{E}m_0(r)^{p/r}]^n<\infty$ for all $r\in[2, p]$, which implies that $\rho[\mathbb{E}m_0(r)^{p/r}]^{1/p}<1$ for all
$r\in[2, p]$. Taking $r=2,p$, we get $\rho\max\{[\mathbb{E}m_0(p)]^{1/p},[\mathbb{E}m_0(2)^{p/2}]^{1/p}\}<1$.

(ii) The sufficiency. By Burkholder's inequality and Minkowski's inequality,
\begin{eqnarray*}
\sup_n\mathbb{E}|\hat{A}_n|^p&\leq&C\mathbb{E}\left(\sum_{n=0}^{\infty}\rho^{2n}(W_{n+1}-W_n)^2\right)^{p/2}\\
&\leq&C\left(\sum_{n=0}^{\infty}\rho^{2n}(\mathbb{E}|W_{n+1}-W_n|^p)^{2/p}\right)^{p/2}.
\end{eqnarray*}
By (\ref{Lp2}),
\begin{eqnarray}\label{L4.1}
\mathbb{E}|W_{n+1}-W_n|^p&\leq&
C\mathbb{E}|W_1-1|^p\mathbb{E}P_n(2)^{p/2}W_n(2)^{p/2}
=C\mathbb{E}|W_1-1|^pU_n^{(2)}(p/2, p/2).
\end{eqnarray}
Since $\mathbb{E}m_0(2)^{p/2}<\infty$, and
$$\mathbb{E}m_0(2)^{p/2}W_1(2)^{p/2}=\mathbb{E}\left(\sum_{u\in\mathbb T_1}\tilde X_u^2\right)^{p/2}\leq \mathbb{E}\left(\sum_{u\in \mathbb T_1}\tilde X_u\right)^{p}=\mathbb{E}W_1^p<\infty, $$
by Lemma \ref{LPBRWL4.2},
\begin{equation}\label{L4.2}
U_n^{(2)}(p/2, p/2)\leq Cn^\gamma\left[\max\left\{\left(\mathbb{E}m_0(2)^im_0(p-2i)\right)_{0\leq i\leq b, i\in\mathbb{N}}, \mathbb{E}m_0(2)^{p/2}\right\}\right]^n
\end{equation}
for $p/2\in(b+1,b+2]$ ($b\geq 0$ is an integer), where $\gamma=1$ for $b=0$ and $\gamma=bp/2$ for $b\geq 1$. Observing that
$\mathbb{E}m_0(2)^xm_0(p-2x)$ is $\log$ convex (see  Lemma \ref{LPBRWL4.3}), we have
\begin{eqnarray}\label{L4.3}
&&\max\left\{\left(\mathbb{E}m_0(2)^im_0(p-2i)\right)_{ 0\leq i\leq b,i\in\mathbb{N}},\mathbb{E}m_0(2)^{p/2}\right\}\nonumber\\
&\leq&\sup_{0\leq x\leq p/2-1}\{\mathbb{E}m_0(2)^xm_0(p-2x)\}
=\max\{\mathbb{E}m_0(p),\mathbb{E}m_0(2)^{p/2}\}.
\end{eqnarray}
Thus, by ( \ref{L4.1}), (\ref{L4.2}) and (\ref{L4.3}), we have
$$\mathbb{E}|W_{n+1}-W_n|^p\leq Cn^\gamma\left[\max\left\{\mathbb{E}m_0(p),\mathbb{E}m_0(2)^{p/2}\right\}\right]^n.$$
Hence we obtain
\begin{eqnarray*}
\sum_n\rho^{2n}(\mathbb{E}|W_{n+1}-W_n|^p)^{2/p}\leq C\sum_n\rho^{2n}n^{2\gamma/p}\left[\max\left\{[\mathbb{E}m_0(p)]^{2/p},[\mathbb{E}m_0(2)^{p/2}]^{2/p}\right\}\right]^n.
\end{eqnarray*}
The right side is finite if and only if $\rho\max\{[\mathbb{E}m_0(p)]^{1/p},[\mathbb{E}m_0(2)^{p/2}]^{1/p}\}<1$.
\end{proof}

\bigskip

Now we prove Theorems \ref{LPBRW1.3} and \ref{LPBRW1.4}, using the moment results of $\hat A_n$.

\begin{proof}[Proof of Theorem \ref{LPBRW1.3}]
(i) The sufficiency. For $1<p<2$, applying Proposition \ref{LPBRWT5.1} with $\rho=1$, we obtain $\sup_n\mathbb E W_n^p<\infty$, which is equivalent to $W_n\rightarrow W$ in $\mathbb P$-$L^p$. For $p\geq2$, by the log convexity of $m_0(x)$ and Jensen's inequality, one has
$$\mathbb E m_0(2)^{p/2}\leq \mathbb E m_0(p)^{p/2(p-1)}\leq\left[ \mathbb E m_0(p)\right]^{p/2(p-1)}.$$
So the condition  $\mathbb E m_0(p)<1$ ensures $\max\{[\mathbb{E}m_0(p)]^{1/p},[\mathbb{E}m_0(2)^{p/2}]^{1/p}\}<1$.
Applying Proposition \ref{LPBRWT5.2} with $\rho=1$ yields the results.

(ii) The necessity. Notice that $\sup_n\mathbb{E}W_n^p<\infty$ ensures that $\mathbb{E}W_1^p<\infty$ and $0<\mathbb{E}W^p<\infty$.
One can see that $W$ satisfies the equation
$$W=\sum_{u\in\mathbb T_1}\tilde X_uW_{(u)}\qquad a.s.,$$
where $W_{(u)}$ denotes the limit random variable of the martingale $W_{n,u}$, and the distribution of $W_{(u)}$ is
$\mathbb{P}_\xi(W_{(u)}\in\cdot)=\mathbb{P}_{T^{|u|}\xi}(W\in\cdot)$. We have
$$W^p=\left(\sum_{u\in\mathbb T_1}\tilde X_uW_{(u)}\right)^p\geq\sum_{u\in\mathbb T_1}\tilde X_u^pW_{(u)}^p\quad a.s.,$$
and the strict inequality holds with positive probability. Thus
$$\mathbb{E}W^p>\mathbb{E}\sum_{u\in\mathbb T_1}\tilde X_u^pW_{(u)}^p=\mathbb{E}m_0(p)\mathbb{E}W^p,$$
which implies that $\mathbb{E}m_0(p)<1$.

\end{proof}

\begin{proof}[Proof of Theorem \ref{LPBRW1.4}]
The assertion (a) is  consequently from  Proposition \ref{LPBRWT5.1} with $\rho>1$ and
 the assertion (b) is from  Proposition \ref{LPBRWT5.2}  with $\rho>1$.
\end{proof}

\section{Uniform convergence; Proof of Theorem \ref{LPBRWU1}}\label{LPBRWS4}

In this section, we study the uniform convergence of the martingale $W_n(t)$, regarding $W_n(t)$ as the function of $t$ (so $t$ is not fixed). Here we just consider the quenched uniform convergence and give the proof of Theorem \ref{LPBRWU1}. The annealed uniform convergence can be obtain almost in the same way, so  we omit the proof of Theorem \ref{LPBRWU2}. The basic tool is still the inequalities for martingale. But in contrast  to the convergence for $t$ fixed, we should consider the superior on a interval of $t$ while estimating the moment of $W_n(t)$.  We first provide two related lemmas.

\begin{lem}\label{UL1}
Let $D=[t_1,t_2]\subset I$. If $D\subset [0,\infty)$ or $D\subset(-\infty, 0]$, then there exists a constant $p_D>1$ such that for any $p\in(1,p_D]$,
$$\mathbb E\log \sup_{t\in D}\left[\frac{m_0(pt)}{m_0(t)^p}\right]<0.$$

\end{lem}

\begin{proof}
For $t_0\in\mathbb R$, set $g_{t_0}(p)=\Lambda(p t_0)-p \Lambda (t_0)$, whose derivative is
$$g_{t_0}'(p)=t_0\Lambda'(p t_0)-\Lambda (t_0).$$
If $t_0\in I$, then $g_{t_0}'(1)=t_0\Lambda'( t_0)-\Lambda (t_0)<0$. Hence there exists a $p_0>1$ such that $g_{t_0}(p)$ is strictly decreasing on $(1,p_0]$, so that $g_{t_0}(p)<g_{t_0}(1)=0$ for all $p\in(1,p_0]$.

Assume $D\subset [0,\infty)$. Since $t_2\in I$, there exists $p_2>1$ such that $g_{t_2}(p)<0$ for all $p\in(1,p_2]$. For $p>1$ fixed, set $f_p(t)=\log m_0(p t)-p\log m_0(t)$. Clearly, $\mathbb Ef_p(t)=g_t(p)$. The derivative of  $f_p(t)$ is
$$f'_p(t)=p\left(\frac{m_0'(p t)}{m_0(p t)}-\frac{m_0'( t)}{m_0(t)}\right).$$
By the convexity of $\log m_0(t)$, we see that the function $\frac{m_0'( t)}{m_0(t)}$ is increasing, so that $f'_p(t)\geq0$ for $t\geq0$ and $f'_p(t)\leq0$ for $t\leq0$. Thus $f_p(t)$ is increasing on $[0,\infty)$ and decreasing on $(-\infty, 0]$. Since $D\subset [0,\infty)$, we have $$\sup\limits_{t\in D}f_p (t)=f_p (t_2).$$
Take $p_D=p_2$. Then for any $p\in (1, p_D]$,
$$\mathbb E\log \sup_{t\in D}\left[\frac{m_0(pt)}{m_0(t)^p}\right]=\mathbb E\sup_{t\in D}\log \left[\frac{m_0(pt)}{m_0(t)^p}\right]=\mathbb E\sup\limits_{t\in D}f_p (t)=\mathbb E f_p (t_2)=g_{t_2}(p)<0.$$
For $D\subset (-\infty, 0]$, the proof is similar.
\end{proof}

\begin{re} \emph{Set $G_{t_0}(p)=\mathbb E\left[\frac{m_0(pt)}{m_0(t)^p}\right]$ for $t_0\in\mathbb R$ fixed. If $t_0\in I'$, then there exists a $p_0>1$ such tat $G_{t_0}(p_0)<1$. The log convexity of $G_{t_0}(p)$ yields $G_{t_0}(p)<1$ for any $p\in (1, p_0]$. Notice that $G_t(p)=\mathbb E e^{f_p(t)}$. With similar arguments to the proof of Lemma \ref{UL1}, we can obtain that if $D=[t_1,t_2]\subset I'$, then there exists a constant $p_D>1$ such that for any $p\in(1,p_D]$,
$$\sup_{t\in D}\mathbb E\left[\frac{m_0(pt)}{m_0(t)^p}\right]<1.$$
This result could be used to study of the annealed uniform convergence in the role of  replacing Lemma \ref{UL1} for quenched case.}
\end{re}

\bigskip

Recall that $\underline m_0=\inf\limits_{t\in I}m_0(t)$ and $\underline m=essinf \underline m_0$.

\begin{lem}\label{UL2}
Let $D=[t_1,t_2]\subset I\bigcap\Omega_1$. If $\mathbb E\log^- \underline m_0<\infty$, then  there exists a constant $p_D>1$ such that
$$\mathbb E \log \sup_{t\in D}\mathbb E_\xi W_1(t)^{p_D}<\infty.$$
\end{lem}

\begin{proof}Since $D\subset I$, we have $\inf\limits_{t\in D}m_0(t)\geq \underline m_0$. Notice that for all $t\in D$,
\begin{eqnarray}
W_1(t)=\sum_{u\in\mathbb T_1}\frac{e^{tS_u}}{m_0(t)}
&\leq&\frac{1}{\underline m_0}\left(\sum_{u\in\mathbb T_1}e^{t_2S_u}\mathbf 1_{\{S_u\geq0\}}+\sum_{u\in\mathbb T_1}e^{t_1S_u}\mathbf 1_{\{S_u<0\}}\right)\nonumber\\
&\leq&\frac{1}{\underline m_0}\left(\tilde Z_1(t_1)+\tilde Z_1(t_2)\right)
\end{eqnarray}
Thus $\mathbb E \log \sup_{t\in D}\mathbb E_\xi W_1(t)^{p}<\infty$ if $\mathbb E\log ^-\underline m_0<\infty$ and $\mathbb E\log^+\mathbb E\tilde Z_1(t_i)^p<\infty$ ($i=1,2$). Since $t_i\in\Omega_1$, there exists $p_i>1$ such that $\mathbb E\log \mathbb E_\xi W_1(t_i)^{p_i}<\infty$, which is equivalent to $\mathbb E\log^+ \mathbb E_\xi \tilde Z_1(t_i)^{p_i}<\infty$ under condition $\mathbb E|\log m_0(t_i)|<\infty$. Taking $p_D=\min\{p_1, p_2\}$ completes the proof.

\end{proof}

\begin{lem}\label{UL3}
Let $D=[t_1,t_2]\subset I\bigcap\Omega_2$ and $W_D^*=\sup\limits_{t\in D}W_1(t)$. If $\underline m>0$, then
$\mathbb E W_D^*\log^+W_D^* <\infty$.
\end{lem}

\begin{proof}
By (4.1), we see that $$W_D^*\leq \frac{1}{\underline m}\left(\tilde Z_1(t_1)+\tilde Z_1(t_2)\right).$$
Since $t_i\in \Omega_2$ ($i=1,2$), we have $\mathbb E \tilde Z_1(t_i)\log^+\tilde Z_1(t_i)<\infty$, which ensures $\mathbb E W_D^*\log^+W_D^* <\infty$.
\end{proof}

\bigskip
Now we prove Theorem \ref{LPBRWU1}.

\begin{proof}[Proof of Theorem  \ref{LPBRWU1}]
We first consider the assertion (a). Clearly, it suffices to  prove that for each $t_0\in I\bigcap\Omega_1$, there exists an interval $D=[t_0-\varepsilon, t_0+\varepsilon]\subset I\bigcap\Omega_1$ ($\varepsilon>0$ small enough) such that the series
\begin{equation}\label{UPE1}
\sum_n \sup_{t\in D}\left(\mathbb E_\xi |W_{n+1}(t)-W_n(t)|^p\right)^{1/p}<\infty \qquad a.s.
\end{equation}
for suitable $1<p\leq 2$. By (\ref{que1}),  we have a.s.,
\begin{eqnarray}\label{UPE2}
\sup_{t\in D}\left(\mathbb E_\xi |W_{n+1}(t)-W_n(t)|^p\right)^{1/p}
&\leq& C\frac{P_n(pt)^{1/p}}{P_n(t)}\left(\mathbb E_{T^n\xi}
|W_1(t)-1|^p\right)^{1/p}\nonumber\\
&\leq& C\left(\sup_{t\in D}\left[\frac{P_n(pt)}{P_n(t)^p}\right]\right)^{1/p}\left(\sup_{t\in D}\mathbb E_{T^n\xi}|W_1(t)-1|^p\right)^{1/p}.
\end{eqnarray}
Here and after the general constant $C$ does not depend on $t$. Decompose $D=D^+\bigcup D^-$, where $D^+=D\bigcap[0,\infty)$ and $D^-=(-\infty, 0]$. Then
\begin{eqnarray}\label{UPE3}
\sup_{t\in D}\left[\frac{P_n(pt)}{P_n(t)^p}\right]&\leq&\max\left\{\sup_{t\in D^+}\left[\frac{P_n(pt)}{P_n(t)^p}\right],\sup_{t\in D^-}\left[\frac{P_n(pt)}{P_n(t)^p}\right]\right\}\nonumber\\
&\leq&\max\left\{\prod_{i=0}^{n-1}\sup_{t\in D^+}\left[\frac{m_i(pt)}{m_i(t)^p}\right],\prod_{i=0}^{n-1}\sup_{t\in D^-}\left[\frac{m_i(pt)}{m_i(t)^p}\right]\right\}
\end{eqnarray}
By ergodic theorem and Lemma \ref{UL1}, a.s.,
\begin{equation}\label{UPE4}
\lim_{n\rightarrow\infty}\frac{1}{n}\log\left(\prod_{i=0}^{n-1}\sup_{t\in D^+}\left[\frac{m_i(pt)}{m_i(t)^p}\right]\right)=\mathbb E\log \sup_{t\in D^+}\left[\frac{m_0(pt)}{m_0(t)^p}\right]<0
\end{equation}
for suitable $1<p\leq2$. The same is true with $D^+$ replaced by $D^-$. Thus, there exists a constant $a_D>1$ such that
\begin{equation}\label{UPE5}
\max\left\{\prod_{i=0}^{n-1}\sup_{t\in D^+}\left[\frac{m_i(pt)}{m_i(t)^p}\right],\prod_{i=0}^{n-1}\sup_{t\in D^-}\left[\frac{m_i(pt)}{m_i(t)^p}\right]\right\}<a_D^{-n}\qquad a.s.
\end{equation}
for $n$ large enough. Combing (\ref{UPE3}) and (\ref{UPE5}) yields
\begin{equation}\label{UPE5}
\sup_{t\in D}\left[\frac{P_n(pt)}{P_n(t)^p}\right]<a_D^{-n}\qquad a.s.
\end{equation}
for $n$ large enough. Hence the a.s. convergence of the series
\begin{equation}\label{UPE6}
\sum_n a_D^{-n}\left(\sup_{t\in D}\mathbb E_{T^n\xi}|W_1(t)-1|^p\right)^{1/p}
\end{equation}
implies (\ref{UPE1}). Since $a_D>1$ and $\mathbb E \log \sup\limits_{t\in D}\mathbb E_\xi W_1(t)^{p}<\infty$, the a.s. convergence of (\ref{UPE6}) is ensured by Lemma \ref{LPBRWL2.1}.

We next prove the assertion (b). For $t_0\in I\bigcap\Omega_2$,  take $\varepsilon>0$ small enough such that the series
\begin{equation}\label{UPEE1}
\sum_n \sup_{t\in D}\mathbb E_\xi |W_{n+1}(t)-W_n(t)|<\infty \qquad a.s..
\end{equation}
We will use a truncation method, similarly to Biggins \cite{b91}. Set  $I_n=\mathbf 1_{\{|W_1(t)-1|\geq c^n\}}$ and $\bar I_n=1-I_n$, where $c>1$
is a constant whose value will be taken later. Using (\cite{b91}, Lemma 4), we get for $1\leq p\leq2$,
\begin{equation*}
\mathbb E_\xi |W_{n+1}(t)-W_n(t)|\leq C\left(\mathbb E_{T^n\xi}|W_1(t)-1|I_n+\left(\mathbb E_{T^n\xi}|W_1(t)-1|^p\bar I_n\right)^{1/p}\left[\frac{P_n(pt)}{P_n(t)^p}\right]^{1/p}\right).
\end{equation*}
To get (\ref{UPEE1}), we  need to consider the a.s. convergence of the two series:
\begin{equation}\label{UPEE2}
\sum_n\sup_{t\in D}\mathbb E_{T^n\xi}|W_1(t)-1|I_n,
\end{equation}
\begin{equation}\label{UPEE3}
\sum_n\left(\sup_{t\in D}\mathbb E_{T^n\xi}|W_1(t)-1|^p\bar I_n\right)^{1/p}\left(\sup_{t\in D}\left[\frac{P_n(pt)}{P_n(t)^p}\right]\right)^{1/p}.
\end{equation}
For (\ref{UPEE2}), observe that
$$\sup_{t\in D}\mathbb E_{T^n\xi}|W_1(t)-1|I_n\leq\mathbb E_{T^n\xi}(W^*_D+1)\mathbf 1_{\{W^*_D+1\geq c^n\}}.$$
By Lemma (\ref{UL3}), we see  $\mathbb E W_D^*\log^+W_D^* <\infty$. Thus,
\begin{eqnarray*}
\mathbb E\left(\sum_n\mathbb E_{T^n\xi}(W^*_D+1)\mathbf 1_{\{W^*_D+1\geq c^n\}}\right)&=&\sum_n\mathbb E(W^*_D+1)\mathbf 1_{\{W^*_D+1\geq c^n\}}\\
&\leq &C \mathbb E(W^*_D+1)\log (W^*_D+1)<\infty,
\end{eqnarray*}
which leads to the a.s. convergence of (\ref{UPEE2}). For (\ref{UPEE3}), Notice (\ref{UPE6}) and the fact that
$$\sup_{t\in D}\mathbb E_{T^n\xi}|W_1(t)-1|^p\bar I_n\leq c^{np}.$$
Taking $1<c<a_D$ yields the a.s. convergence of (\ref{UPEE2}). The proof is completed.
\end{proof}


\section{Moderate deviation principles }\label{LPBRWS5}

\subsection{Moderate deviation principle for $\frac{\mathbb{E}_\xi Z_n(a_n\cdot)}{\mathbb{E}_\xi Z_n(\mathbb{R})}$; Proof of Theorem \ref{MDP}}
We first study the moderate deviations for the quenched means.

\begin{thm}[Moderate deviation principle for quenched means $\frac{\mathbb{E}_\xi Z_n(a_n\cdot)}{\mathbb{E}_\xi Z_n(\mathbb{R})}$]\label{LTT2.5.1} Write $\pi_0=m_0(0)$.
If  $\|\frac{1}{\pi_0}\mathbb{E}_\xi\sum\limits_{i=1}^Ne^{\delta |L_i|}\|_\infty:=esssup\frac{1}{\pi_0}\mathbb{E}_\xi\sum\limits_{i=1}^Ne^{\delta |L_i|}<\infty$ for some
$\delta>0$ and $\mathbb{E}_\xi\sum\limits_{i=1}^NL_i=0\;a.s.$,  then the sequence of probability measures
$A \mapsto  \frac{\mathbb{E}_\xi Z_n (a_nA)}{\mathbb{E}_\xi Z_n (\mathbb{R})}$ satisfies a principle of moderate deviation: for each measurable subset $A$ of
$\mathbb{R}$,

\begin{eqnarray}
-\frac{1}{2{\sigma}^2}\inf_{x\in A^o
}x^2&\leq&\liminf_{n\rightarrow\infty}\frac{n}{a_n^2}\log
\frac{\mathbb{E}_\xi Z_n(a_nA)}{\mathbb{E}_\xi Z_n}\notag \\
&\leq&\limsup_{n\rightarrow\infty}\frac{n}{a_n^2}\log
\frac{\mathbb{E}_\xi Z_n(a_nA)}{\mathbb{E}_\xi
Z_n}\leq-\frac{1}{2{\sigma}^2}\inf_{x\in \bar{A}}x^2\qquad a.s.,\label{LTE2.5.1}
\end{eqnarray}
where ${\sigma}^2=\mathbb{E}\frac{1}{\pi_0}\sum\limits_{i=1}^NL_i^2$,  and $A^o$ denotes the interior of $A$ and $\bar A$ its closure.
\end{thm}

\begin{proof}[Proof]
Consider the probability measures
$q_n(\cdot)=\frac{\mathbb{E}_\xi Z_n(a_n\cdot)}{\mathbb{E}_\xi Z_n(\mathbb{R})}$.
Let
$$\lambda_n(t)=\log\int e^{tx}q_n(dx)=\log \left[\frac{\mathbb{E}_\xi\int e^{a_n^{-1}tx}Z_n(dx)}{\mathbb{E}_\xi Z_n(\mathbb{R})}\right].$$
Then
\begin{equation}\label{MDE1}
\lambda_n(t)=\log\left[\frac{P_n(a_n^{-1}t)}{P_n(0)}\right]=\sum_{i=0}^{n-1}\log \left[\frac{m_i(a_n^{-1}t)}{\pi_i}\right].
\end{equation}
Set $\lambda(t)=\frac{1}{2}{\sigma}^2t^2$, whose
Legendre-transform  is
$\lambda^*(x)=\sup_{x\in{\mathbb{R}}}\{tx-\lambda(t)\}=\frac{x^2}{2{\sigma}^2}$. We shall show that for each $t\in\mathbb R$,
\begin{equation}\label{LTE2.5.3}
\lim_{n\rightarrow\infty}\frac{n}{a_n^2}\lambda_n(\frac{a_n^2}{n}t)=\lambda(t)\qquad a.s..
\end{equation}
Then (\ref{LTE2.5.3}) a.s. holds for all rational $t$, and hence for all $t\in\mathbb R$ by the convexity of $\lambda_n(t)$ and the continuity of $\lambda(t)$. By the G\"{a}rtner-Ellis theorem (cf. \cite{z}, p52, Exercises 2.3.20), we get (\ref{LTE2.5.1}).

Put $\Delta_{n,i}=\frac{m_i(\frac{a_n}{n}t)}{\pi_i}-1$. We will see that  for each $t\in\mathbb R$,
\begin{equation}\label{MDE2}
\sup_i|\Delta_{n,i}|<1\qquad a.s.
\end{equation}
 for $n$ large enough.
Denote $M=\|\frac{1}{\pi_0}\mathbb{E}_\xi\sum\limits_{i=1}^Ne^{\delta |L_i|}\|_\infty$. Then $\sup_n\frac{1}{\pi_n}\mathbb{E}_\xi\int e^{\delta|x|}X_n(dx)\leq M$ a.s., so that  for $n$ large enough,
$$\sum_{k=0}^{\infty}\frac{1}{\pi_i}\mathbb{E}_\xi\int\frac{1}{k!}\left|\frac{a_n}{n}tx\right|^kX_i(dx)\leq M \qquad a.s..$$
Notice that $\frac{1}{\pi_n}\mathbb{E}_\xi\int xX_n(dx)=\frac{1}{\pi_n}\mathbb{E}_\xi\sum\limits_{i=1}^NL_i=0$ a.s..
Therefore a.s.,
\begin{eqnarray}\label{LTE2.5.4}
\Delta_{n,i}&=&\frac{1}{\pi_i}\mathbb{E}_\xi\int e^{a_nn^{-1}tx}X_i(dx)-1\nonumber\\
&=&\frac{1}{\pi_i}\mathbb{E}_\xi\int\left(\sum_{k=0}^{\infty}\frac{1}{k!}\left(\frac{a_n}{n}tx\right)^k\right)X_i(dx)-1\nonumber\\
&=&\sum_{k=0}^{\infty}\frac{1}{k!}\left(\frac{a_n}{n}t\right)^k\frac{1}{\pi_i}\mathbb{E}_\xi\int x^kX_i(dx)-1\nonumber\\
&=&\sum_{k=2}^{\infty}\frac{1}{k!}\left(\frac{a_n}{n}t\right)^k\frac{1}{\pi_i}\mathbb{E}_\xi\int x^kX_i(dx).
\end{eqnarray}
It follows that
\begin{eqnarray}\label{LTE2.5.2}
\sup_i|\Delta_{n,i} |
&\leq&\sup_i\left[\sum_{k=2}^{\infty}\frac{1}{k!}\left(\frac{a_n}{n}|t|\right)^k\frac{1}{\pi_i}\mathbb{E}_\xi\int |x|^kX_i(dx)\right]\nonumber\\
&\leq&\sup_i\left[\sum_{k=2}^{\infty}\left(\frac{a_n}{n}\frac{|t|}{\delta} \right)^k\frac{1}{\pi_i}\mathbb{E}_\xi\int e^{ \delta| x|}X_i(dx)\right]\nonumber\\
&\leq& M\sum_{k=2}^{\infty}\left(\frac{a_n}{n}\frac{|t|}{\delta} \right)^k
\leq M_1\left(\frac{a_n}{n}\right)^2\qquad a.s.,
\end{eqnarray}
where $M_1>0$ is a constant (it depends on $t$). Hence (\ref{MDE2}) holds  for $n$ large enough.

Now we  calculate the limit (\ref{LTE2.5.3}). By (\ref{MDE1}) and (\ref{MDE2}), we have for $n$ large enough, a.s.,
\begin{eqnarray*}
\frac{n}{a_n^2}\lambda_n(\frac{a_n^2}{n}t)
&=&\frac{n}{a_n^2}\sum_{i=0}^{n-1}\log\left(1+  \Delta_{n,i} \right)\\
&=&\frac{n}{a_n^2}\sum_{i=0}^{n-1}\sum_{j=1}^{\infty}\frac{(-1)^{j+1}}{j}\left( \Delta_{n,i}\right)^j\\
&=&\frac{n}{a_n^2}\sum_{j=1}^{\infty}\frac{(-1)^{j+1}}{j}\sum_{i=0}^{n-1}\left( \Delta_{n,i}\right)^j\\
&=&\frac{n}{a_n^2}\sum_{i=0}^{n-1}\Delta_{n,i}
+\frac{n}{a_n^2}\sum_{j=2}^{\infty}\frac{(-1)^{j+1}}{j}\sum_{i=0}^{n-1}\left( \Delta_{n,i}\right)^j\\
&=&:A_n+B_n.
\end{eqnarray*}
For $B_n$, by (\ref{LTE2.5.2}),
\begin{eqnarray*}
|B_n|\leq\frac{n}{a_n^2}\sum_{j=2}^{\infty}\frac{1}{j}\sum_{i=0}^{n-1}\left|\Delta_{n,i}\right|^j
\leq\sum_{j=2}^{\infty}M_1^j\left(\frac{a_n}{n}\right)^{2j-2}
\leq M_2\left(\frac{a_n}{n}\right)^2\rightarrow0\quad a.s.\quad \text{as}\;n\rightarrow\infty,
\end{eqnarray*}
where $M_2>0$ is a constant.  For  $A_n$, by (\ref{LTE2.5.4}), a.s.,
\begin{eqnarray*}
A_n&=&\frac{n}{a_n^2}\sum_{i=0}^{n-1}\sum_{k=2}^{\infty}\frac{1}{k!}\left(\frac{a_n}{n}t\right)^k\frac{1}{\pi_i}\mathbb{E}_\xi\int x^kX_i(dx)\\
&=&\frac{n}{a_n^2}\sum_{k=2}^{\infty}\frac{1}{k!}\left(\frac{a_n}{n}t\right)^k\sum_{i=0}^{n-1}\frac{1}{\pi_i}\mathbb{E}_\xi\int x^kX_i(dx)\\
&=&\frac{n}{a_n^2}\frac{1}{2}\left(\frac{a_n}{n}t\right)^2\sum_{i=0}^{n-1}\frac{1}{\pi_i}\mathbb{E}_\xi\int x^2X_i(dx)\\
&&+\frac{n}{a_n^2}\sum_{k=3}^{\infty}\frac{1}{k!}\left(\frac{a_n}{n}t\right)^k\sum_{i=0}^{n-1}\frac{1}{\pi_i}\mathbb{E}_\xi\int x^kX_i(dx)\\
&=&:C_{n}+D_{n}.
\end{eqnarray*}
The ergodic theorem gives
\begin{eqnarray*}
\lim_{n\rightarrow\infty}C_{n}=\frac{1}{2}t^2\lim_{n\rightarrow\infty}\frac{1}{n}\sum_{i=0}^{n-1}\frac{1}{\pi_i}\mathbb{E}_\xi\int x^2X_i(dx)=\frac{1}{2}{\sigma}^2t^2=\lambda(t)\quad a.s..
\end{eqnarray*}
To get (\ref{LTE2.5.3}), it remains to show that $D_{n}$ is negligible. Clearly, a.s.
\begin{eqnarray*}
|D_{n}|&\leq&\frac{n}{a_n^2}\sum_{k=3}^{\infty}\frac{1}{k!}\left(\frac{a_n}{n}\frac{|t|}{\delta}\right)^k
\sum_{i=0}^{n-1}\frac{1}{\pi_i}\mathbb{E}_\xi\int |\delta x|^kX_i(dx)\\
&\leq& M\sum_{k=3}^{\infty}\left(\frac{a_n}{n}\frac{|t|}{\delta} \right)^{k-2}\leq M_3\frac{a_n}{n}
\rightarrow0\qquad a.s.\qquad \text{as}\;n\rightarrow\infty,
\end{eqnarray*}
where $M_3>0$ is a constant. This completes the proof.
\end{proof}

The moderate deviation principle for $\frac{Z_n(a_n\cdot)}{Z_n(\mathbb R)}$ comes from Theorem \ref{LTT2.5.1} and the uniform convergence of $W_n(t)$ (Theorem \ref{LPBRWU1}).

\begin{proof}[Proof of Theorem \ref{MDP}]
Let
$$\Gamma_n(t)=\log\left[\frac{\int e^{a_n^{-1}tx}Z_n(dx)}{Z_n(\mathbb R)}\right]=\log\left[\frac{\tilde Z_n(a_n^{-1}t)}{Z_n(\mathbb R)}\right].$$
Notice that
\begin{equation}\label{MEP1}
\frac{n}{a_n^2}\Gamma_n(\frac{a_n^2}{n}t)=\frac{n}{a_n^2}\log W_n(\frac{a_n}{n}t)+\frac{n}{a_n^2}\lambda_n(\frac{a_n^2}{n}t)
-\frac{n}{a_n^2}\log W_n(0).
\end{equation}
As $m'_0(0)=\mathbb E_\xi \sum\limits_{i=1}^N L_i=0$, the log convexity of $m_0(x)$ gives $\underline m_0=\pi_0$.
Since $\mathbb E|\log \pi_0|<\infty$ and $0\in I\bigcap\Omega_1$, by Theorem \ref{LPBRWU1}, $W_n(t)$ converges uniformly  a.s. to $W(t)$ on $[-\varepsilon, \varepsilon]$ for some $\varepsilon>0$, so that $W_n(t)$ is continuous at $0$. Thus
$$W_n(\frac{a_n}{n}t)\rightarrow W(0)\qquad a.s. \quad\text{as $n\rightarrow\infty$}.$$
The assumption (\ref{ASS}) implies that $W(0)>0$ a.s. on $\{Z_n(\mathbb R)\rightarrow\infty\}$. Letting $n\rightarrow\infty$ and using (\ref{LTE2.5.3}), we obtain for each $t\in\mathbb R$,
\begin{equation}\label{MEP2}
\lim_{n\rightarrow}\frac{n}{a_n^2}\Gamma_n(\frac{a_n^2}{n}t)
=\lambda(t)=\frac{1}{2}\sigma^2t^2\quad \text{a.s. on $\{Z_n(\mathbb R)\rightarrow\infty\}$. }
\end{equation}
So (\ref{MEP2}) a.s. holds for all rational $t$, and therefore for all $t\in\mathbb R$ by the convexity of $\Gamma_n(t)$ and the continuity of $\lambda(t)$. Then apply the G\"{a}rtner-Ellis theorem.

\end{proof}

\subsection{Moderate deviation principles for $\mathbb E\frac{ Z_n(a_n\cdot)}{\mathbb{E}_\xi Z_n(\mathbb{R})}$ and $\frac{\mathbb{E}Z_n(a_n\cdot)}{\mathbb{E}Z_n(\mathbb{R})}$}
Finally, in i.i.d environment, we also have  moderate deviation principles for $\mathbb E\frac{ Z_n(a_n\cdot)}{\mathbb{E}_\xi Z_n(\mathbb{R})}$ and $\frac{\mathbb{E}Z_n(a_n\cdot)}{\mathbb{E}Z_n(\mathbb{R})}$.

\begin{thm}[Moderate deviation principle for $\mathbb{E}\frac{Z_n(a_n\cdot)}{\mathbb{E}_\xi Z_n(\mathbb{R})}$]\label{LTT2.5.2}
Assume that $\xi_n$ are i.i.d.. Write $\pi_0=m_0(0)$. If $\mathbb{E}\frac{1}{\pi_0}\sum\limits_{i=1}^Ne^{\delta |L_i|}<\infty$ for some $\delta<0$ and $\mathbb{E}\frac{1}{\pi_0}\sum\limits_{i=1}^NL_i=0$,
 then the sequence of finite measures
$A \mapsto  \mathbb{E}\frac{Z_n (a_nA)}{\mathbb{E}Z_n(\mathbb{R})}$ satisfies a principle of moderate deviation: for each measurable subset $A$ of
$\mathbb{R}$,
\begin{eqnarray}\label{LTE2.5.6}
-\frac{1}{2{\sigma}^2}\inf_{x\in A^o}x^2&\leq&\liminf_{n\rightarrow\infty}\frac{n}{a_n^2}\log
\mathbb{E}\frac{Z_n(a_nA)}{\mathbb{E}_\xi Z_n(\mathbb{R})}\nonumber\\
&\leq&\limsup_{n\rightarrow\infty}\frac{n}{a_n^2}\log
\mathbb{E}\frac{Z_n(a_nA)}{\mathbb{E}_\xi
Z_n(\mathbb{R})}\leq-\frac{1}{2{\sigma}^2}\inf_{x\in \bar{A}}x^2,
\end{eqnarray}
where ${\sigma}^2=\mathbb{E}\frac{1}{\pi_0}\sum\limits_{i=1}^NL_i^2$,  and $A^o$ denotes the interior of $A$ and $\bar A$ its closure.
\end{thm}

\begin{proof}[Proof]
Let
$$\lambda_n(t)=\log \mathbb{E}\frac{\int e^{a_n^{-1}tx}Z_n(dx)}{\mathbb{E}_\xi Z_n(\mathbb{R})}\quad\text{and}\quad\lambda(t)=\frac{1}{2}{\sigma}^2t^2.$$
It suffices to show that for all $t\in\mathbb R$,
\begin{equation*}
\lim_{n\rightarrow\infty}\frac{n}{a_n^2}\lambda_n(\frac{a_n^2}{n}t)=\lambda(t).
\end{equation*}
Then (\ref{LTE2.5.6}) holds by  the G\"{a}rtner-Ellis theorem. In fact,
the condition $\mathbb{E}\frac{1}{\pi_0}\sum\limits_{i=1}^Ne^{\delta |L_i|}<\infty$ gives
$$\mathbb E\frac{1}{\pi_0}\int e^{a_nn^{-1}|tx|}X_0(dx)<\infty\quad\text{for $n$ large
enough.}$$  Thus
\begin{eqnarray*}
\lambda_n(\frac{a_n^2}{n}t)&=&n\log \mathbb E\left[\frac{m_0(\frac{a_n}{n}t)}{\pi_0}\right]\\
&=&n\log \mathbb E\frac{1}{\pi_0}\int e^{a_nn^{-1}tx}X_0(dx)\\
&=&n\log \mathbb E\frac{1}{m_0}\int\left(\sum_{k=0}^{\infty}\frac{1}{k!}\left(\frac{a_n}{n}t\right)^kx^k\right)X_0(dx)\\
&=&n\log\left(1+\sum_{k=2}^{\infty}\frac{1}{k!}\left(\frac{a_n}{n}t\right)^k\mathbb E\frac{1}{\pi_0}\int
x^k X_0(dx)\right)\\
&=&n\log\left(1+\frac{1}{2}\left(\frac{a_n}{n}\right)^2t^2{\sigma}^2+o(\left(\frac{a_n}{n}\right)^2)\right).
\end{eqnarray*}
Therefore,
$$\lim_{n\rightarrow\infty}\frac{n}{a_n^2}\lambda_n(\frac{a_n^2}{n}t)
=\lim_{n\rightarrow\infty}\frac{n^2}{a_n^2}
\log\left(1+\frac{1}{2}\left(\frac{a_n}{n}\right)^2t^2\tilde{\sigma}^2+o(\left(\frac{a_n}{n}\right)^2)\right)
=\frac{1}{2}{\sigma}^2t^2.$$
\end{proof}

If we consider the probability measures $\frac{\mathbb{E}Z_n(a_n\cdot)}{\mathbb{E}Z_n(\mathbb{R})}$,
then by  a similar argument to the proof of Theorem \ref{LTT2.5.2}, we obtain the following theorem.

\begin{thm}[Moderate deviation principle for annealed means $ \frac{\mathbb{E}Z_n (a_n\cdot)}{\mathbb{E}Z_n(\mathbb{R})}$]\label{LTT2.5.3}
Assume that $\xi_n$ are i.i.d..  Write $\pi_0=m_0(0)$. If $\mathbb{E} \sum\limits_{i=1}^Ne^{\delta |L_i|}<\infty$ for some $\delta<0$ and $\mathbb{E}\sum\limits_{i=1}^NL_i=0$,
 then the sequence of finite measures
$A \mapsto  \frac{\mathbb{E}Z_n (a_nA)}{\mathbb{E}Z_n(\mathbb{R})}$ satisfies a principle of moderate deviation: for each measurable subset $A$ of
$\mathbb{R}$,
\begin{eqnarray}\label{LTE2.5.7}
-\frac{1}{2\tilde{\sigma}^2}\inf_{x\in A^o}x^2&\leq&\liminf_{n\rightarrow\infty}\frac{n}{a_n^2}\log
\frac{\mathbb{E}Z_n(a_nA)}{\mathbb{E}_\xi Z_n(\mathbb{R})}\nonumber\\
&\leq&\limsup_{n\rightarrow\infty}\frac{n}{a_n^2}\log
\frac{\mathbb{E}Z_n(a_nA)}{\mathbb{E}_\xi
Z_n(\mathbb{R})}\leq-\frac{1}{2\tilde{\sigma}^2}\inf_{x\in \bar{A}}x^2,
\end{eqnarray}
where $\tilde{\sigma}^2=\frac{1}{\mathbb{E}m_0}\mathbb{E}\sum\limits_{i=1}^NL_i^2$,  and $A^o$ denotes the interior of $A$ and $\bar A$ its closure.
\end{thm}

\end{document}